\documentclass[final]{siamltex}

\usepackage{booktabs}
\usepackage[mathscr]{euscript}
\usepackage{color,amsmath,amssymb,mathrsfs}
\usepackage{graphicx}
\usepackage{algorithmic,algorithm}
\usepackage{tikz,pgfplots}
\usepackage{upgreek}
\usepackage{showkeys,cite}
\usepackage{multirow}
\usepackage{yfonts}
\usepackage{mathtools}
\usepackage[normalem]{ulem}
\usepackage{latexsym}
\DeclareMathAlphabet{\mathpzc}{OT1}{pzc}{m}{it}
\usepackage{url} 
\usepackage{tikz} 

\usepackage[plainpages=false,colorlinks=true,citecolor=blue,
 filecolor=blue,linkcolor=blue,urlcolor=blue]{hyperref}

\definecolor{grassgreen}{RGB}{92,135,39}
\newtheorem{assumption}[theorem]{Assumption}
\newcommand{\hilb}{\mathscr{H}}
\newcommand{\defeq}{\vcentcolon=}

\renewcommand{\vec}[1]{{\mathchoice
                     {\mbox{\boldmath$\displaystyle{#1}$}}
                     {\mbox{\boldmath$\textstyle{#1}$}}
                     {\mbox{\boldmath$\scriptstyle{#1}$}}
                     {\mbox{\boldmath$\scriptscriptstyle{#1}$}}}}

\newcommand{\ran}{\mathsf{range}}
\newcommand{\dom}{\mathsf{dom}}
\newcommand{\trace}{\mathsf{tr}}
\newcommand{\eps}{\varepsilon}
\newcommand{\norm}[1]{\left\| {#1} \right\|}
\newcommand{\ip}[2]{{\left\langle {#1}, {#2} \right\rangle}}
\newcommand{\mip}[2]{\left\langle{#1}, {#2}\right\rangle_{\!\scriptscriptstyle{\text{M}}}}
\renewcommand{\grad}{\nabla}
\newcommand{\mat}[1]{\mathbf{{#1}}}
\newcommand\restr[2]{{
  \left.\kern-\nulldelimiterspace % automatically resize the bar with \right
  {#1}\vphantom{\big|} \right|_{#2}}}

\newcommand{\R}{\mathbb{R}}

\newcommand{\N}{\mathbb{N}}

\newcommand{\A}{\mathcal{A}}
\newcommand{\B}{\mathcal{B}}
\newcommand{\C}{\mathcal{C}}
\newcommand{\D}{\mathcal{D}}
\newcommand{\E}{\mathcal{E}}
\newcommand{\F}{\mathcal{F}}

\newcommand{\J}{\mathcal{J}}

\newcommand{\borel}{\mathscr{B}}
\newcommand{\ave}{\mathsf{E}}
\newcommand{\var}{\mathsf{var}}
\newcommand{\GM}[2]{\mathcal{N}\!\left( {#1}, {#2}\right)}

\newcommand{\Cprior}{\mathcal{C}_{\text{pr}}}
\newcommand{\Cpost}{\mathcal{C}_{\text{post}}}
\newcommand{\ncov}{\mat{\Gamma}_{\!\text{noise}}}

\newcommand{\like}{\pi_{\text{like}}}
\newcommand{\obs}{\vec{\mathrm d}}
\newcommand{\ipar}{m}
\newcommand{\iparpr}{m_{\text{pr}}}
\newcommand{\iparmap}{m_{\scriptscriptstyle\text{MAP}}}
\newcommand{\dpar}{\vec{m}}
\newcommand{\dparmap}{\vec{m}_{\scriptscriptstyle\text{MAP}}}

\newcommand{\postcov}{\mat{\Gamma}_{\text{post}} }
\newcommand{\priorm}{\mu_{\text{pr}}}
\newcommand{\postm}{\mu_{\text{post}}^{{\obs}}}
\newcommand{\ff}{\vec{f}}                      % parameter-to-obs MAP
\newcommand{\Ndp}{n'_{\!d}}
\newcommand{\Nw}{{n_w}}

\def\addressices{Institute for Computational Engineering \& Sciences, The
  University of Texas at Austin, Austin, TX, USA}
\def\addressgeo{Department of Geological Sciences, The University of
  Texas at Austin, Austin, TX, USA}
\def\addressmech{Department of Mechanical Engineering, The
  University of Texas at Austin, Austin, TX, USA}

\newcommand{\Ns}{{n_s}}
\newcommand{\Ntr}{{n_{\text{tr}}}}
\newcommand{\Nd}{{n_{\text{d}}}}
\newcommand{\Nm}{n}
\renewcommand{\H}{\mathcal{H}}
\newcommand{\obj}{\Psi}
\newcommand{\objG}{\obj^{\scriptscriptstyle{\text{G}}}}
\newcommand{\LI}{\mathscr{L}^{{{\text{I}}}}}

\newcommand{\LOED}{\mathscr{L}^{{\scriptscriptstyle\text{O}}}}

\newcommand{\eeta}{\vec{\eta}}
\newcommand{\CM}{\mathscr{E}}
\newcommand{\eip}[2]{\left\langle{#1}, {#2}\right\rangle_{\!{\R^q}}}
\newcommand{\cip}[2]{\left\langle{#1}, {#2}\right\rangle_{\!\CM}}
\newcommand{\W}{\mat{W}}
\newcommand{\Wn}{\mat{W}_{\!\sigma}}

\newcommand{\HM}{\H_{\text{misfit}}}
\newcommand{\HMt}{\tilde{\H}_{\text{misfit}}}
\newcommand{\ad}[1]{{#1}^\ast}
\newcommand{\adfd}[1]{\vec{{#1}}^\ast}
\newcommand{\ut}[1]{\ensuremath{\tilde{#1}}}

\newcommand{\Exp}[1]{e^{{#1}}}
\newcommand{\ipart}{m_{\scriptscriptstyle\text{true}}}
\newcommand{\postmGauss}{\mu_{\text{post}}^{{\obs,\text{G}}}}
\newcommand{\commentout}[1]{\iffalse {#1} \fi}

\newcommand{\V}{\mathscr{V}_{\!\scriptscriptstyle{0}}}
\newcommand{\Vg}{\mathscr{V}_{\!g}}
\newcommand{\tab}{\,\,\,\,}
\renewcommand{\O}{\mathcal{O}}

\DeclareMathOperator*{\argmin}{arg\,min}

\newcommand{\GD}{\ensuremath{\Gamma_{\!\!D}}}
\newcommand{\GN}{\ensuremath{\Gamma_{\!\!N}}}
\allowdisplaybreaks
\newcommand{\email}[1]{\protect\href{mailto:#1}{#1}}

\begin{document}

\def\addressncsu{Department of Mathematics, North Carolina State University, Raleigh, NC, USA}

\def\addressnyu{Courant Institute of Mathematical Sciences, New York
  University, New York, NY, USA}

\def\addressucm{Applied Mathematics, School of Natural Sciences, University of California,
  Merced, CA, USA}

\author{Alen Alexanderian\footnotemark[1] \and Noemi Petra\footnotemark[5]
  \and Georg Stadler\footnotemark[6] \and Omar~Ghattas{\footnotemark[2]~\footnotemark[3]~\footnotemark[4]}}
\renewcommand{\thefootnote}{\fnsymbol{footnote}}
\footnotetext[1]{\addressices. Current address:~\addressncsu. Email: \email{alexanderian@ncsu.edu}.}
\footnotetext[2]{\addressices. \email{omar@ices.utexas.edu}.}
\footnotetext[3]{\addressmech}
\footnotetext[4]{\addressgeo}
\footnotetext[5]{\addressucm. Email: \email{npetra@ucmerced.edu}.}
\footnotetext[6]{\addressnyu. Email: \email{stadler@cims.nyu.edu}.}
\renewcommand{\thefootnote}{\arabic{footnote}}

\title{A Fast and Scalable Method for A-Optimal Design of Experiments
  for Infinite-dimensional Bayesian Nonlinear Inverse Problems}

\def \pos {0.49\columnwidth}

\maketitle

\begin{abstract}
We address the problem of optimal experimental design (OED) for
Bayesian nonlinear inverse problems governed by partial differential
equations (PDEs). The inverse problem seeks to infer an
infinite-dimensional parameter
from
experimental data observed at a set of sensor locations and from the
governing PDEs. 
The goal of the OED problem is to find an optimal placement of sensors
so as to minimize the uncertainty in the inferred parameter field.
Specifically, we seek an optimal subset of sensors from among a fixed
set of candidate sensor locations.
We formulate the OED objective function by generalizing the classical
A-optimal experimental design criterion using the expected value of
the trace of the posterior covariance. This expected value is computed
through sample averaging over the set of likely experimental
data. 
To cope with the infinite-dimensional character of the parameter field, 
we construct a Gaussian approximation to the
posterior at the maximum a posteriori probability (MAP) point, and use
the resulting covariance operator to define the OED objective
function. 
We use randomized trace estimation to compute the trace of this
covariance operator, which is defined only implicitly.
The resulting OED problem includes as constraints the system of PDEs
characterizing the MAP point, and the PDEs describing the action of
the covariance (of the Gaussian approximation to the posterior) to
vectors.
We control the sparsity of the sensor configurations using sparsifying
penalty functions.
Variational adjoint methods are used to efficiently compute the
gradient of the PDE-constrained OED objective function. 
We elaborate our OED method for the problem of determining the optimal
sensor configuration to best infer the coefficient of an elliptic
PDE. Furthermore, we provide numerical results for inference of the
log permeability field in a porous medium flow problem.
Numerical results show that the number of PDE solves required for the
evaluation of the OED objective function and its gradient is
essentially independent of both the parameter dimension and the sensor
dimension (i.e., the number of candidate sensor locations). The number
of quasi-Newton iterations for computing an OED also exhibits the same
dimension invariance properties.

\end{abstract}



\section{Introduction}

We address the problem of optimal design of experiments for Bayesian
nonlinear inverse problems governed by partial differential equations
(PDEs).  Our goal is to determine sensor locations, at which
experimental data are collected, in such a way that the uncertainty in
the inferred parameter field is minimized, in a sense made precise
below.  The numerical solution of a Bayesian inverse problem, which is
just a subproblem of the optimal experimental design (OED) problem, is
challenging, in particular for problems with infinite-dimensional
(high-dimensional upon discretization) parameters and
expensive-to-evaluate parameter-to-observable (forward) maps.
Computing optimal experimental designs requires repeated solution of
the underlying Bayesian inverse problem; hence, the OED problem
inherits all of the challenges of solving the Bayesian inverse
problem, which in turn inherits the computational difficulties of
solving the PDEs describing the forward problem.  These challenges
necessitate algorithms that maximally exploit the problem structure to
make OED tractable for problems that are of large scale---in the state,
parameter, and data dimensions.

\paragraph{Related work}
Standard references for OED include~\cite{Ucinski05, AtkinsonDonev92,
  Pukelsheim93,Pazman86}.  While most 
of these classical developments concern OED for inverse problems of
low parameter dimension, and consider well-posed inverse problems,
recently there has been an increased interest in OED for large-scale
problems governed by expensive-to-solve forward models.  In
particular, the authors of
\cite{HaberHoreshTenorio10,HoreshHaberTenorio10, ChungHaber12} present
numerical methods for OED for nonlinear ill-posed inverse problems
governed by large-scale models. In these papers, a frequentist point
of view is taken. In particular, the OED objective function is defined
as an empirical estimate of the Bayes risk of the point
estimator---the solution to a Tikhonov-regularized deterministic
inverse problem---for a finite-dimensional inference parameter. This
amounts to solving an optimization problem for the OED that is
constrained by first-order optimality conditions representing solution
of an inverse problem for each member of a set of training models.
There are two main differences between the work in
\cite{HaberHoreshTenorio10,HoreshHaberTenorio10} and that proposed
here. First, we address the mathematical and computational challenges
stemming from the problem of OED for {\em infinite-dimensional}
inverse problems. In particular, the choice of the prior, of the
discretization, and of the discrete inner products is such that
the discrete problems are all approximations of the same
infinite-dimensional inverse problem.
Second, in the OED objective, we explicitly
incorporate the covariance operator of (a Gaussian approximation of)
the Bayesian posterior measure, thus directly capturing the
uncertainty in the inferred parameters in the objective function. This
entails a more complex and difficult OED optimization problem, since
now it is constrained not only by the first-order optimality
conditions for the inverse problem (i.e., gradients), but also by
second-order information (i.e., Hessians). Nevertheless, we
demonstrate that we can construct scalable algorithms (those whose
cost measured in forward PDE solves is independent of problem
dimension) to solve these OED optimization problems.

Other efforts in the area
include~\cite{BauerBockKorkelEtAl00,KorkelKostinaBockEtAl04}.
In~\cite{BauerBockKorkelEtAl00}, the authors use sequential quadratic
programming (SQP) to compute optimal designs with different OED
criteria for finite-dimensional inverse problems governed by nonlinear systems
of differential--algebraic equations (DAEs).
In~\cite{KorkelKostinaBockEtAl04}, the design of robust
experiments for inverse problems governed by nonlinear DAEs is
addressed; see also the review article~\cite{BockKoerkelSchloeder13}.
While the inverse problems discussed in these papers are governed by
nonlinear DAEs, they usually have a small to moderate number of
parameters.  Another idea, mainly aimed at nonlinear inverse problems
with low to moderate parameter dimension, is that
of~\cite{HuanMarzouk13,HuanMarzouk14} in which the authors use a
generalized polynomial chaos surrogate for the forward model, and
utilize techniques of stochastic optimization to compute experimental
designs that maximize the expected information gain as measured by the
Kullback-Liebler divergence from posterior to prior.  Since no closed
form expression for the expected information gain is available for
nonlinear Bayesian inverse problems, one must resort to
computationally expensive sampling approaches.  The
paper~\cite{LongScavinoTemponeEtAl13} offers an alternate approach 
through a methodology based on a Laplace approximation, i.e., a
Gaussian approximation, of the posterior distribution to accelerate
the numerical computation of the expected information gain. 

\paragraph{Contributions}
In this work we address the OED problem for infinite-dimensional
Bayesian inverse problems,
and seek scalable algorithms for its solution.  
We retain the infinite-dimensional structure of the problem during the
development of solution methods, which not only leads to elegant mathematical
formulations but also is of practical importance: studying the problem in
infinite dimensions guides the choice of prior measures that are meaningful
for infinite-dimensional parameters and forces one to use appropriate
discretizations of the Bayesian inverse problem that avoid mesh artifacts.
Moreover, the infinite-dimensional formulation provides, via the Lagrangian
formalism, a straightforward way to derive adjoint-based expressions for
derivatives of the OED objective.  The main contributions of our work are as
follows: 
(1) We propose a method for A-optimal experimental
design for infinite-dimensional Bayesian nonlinear inverse problems; the
proposed formulation aims at minimizing the expected average posterior
variance. (2) We employ several approximations, which, when combined with
structure-exploiting algorithms, render OED for large-scale inverse problems
computationally tractable. In particular, we formulate the OED problem as a
bilevel PDE-constrained optimization problem.
(3) We use the problem of inferring a coefficient field in an elliptic
PDE to elaborate our approach for A-optimal
sensor placement. For the resulting PDE-constrained OED problem, we
derive efficient adjoint-based expressions for the gradient and assess the
computational complexity of the objective function evaluation and the
gradient computation. (4) We present a comprehensive numerical study
of the effectiveness of the OED method for optimal sensor placement
for a subsurface flow inverse problem and demonstrate scalability of
our framework in terms of the number of forward (and adjoint) PDE
solves as the parameter and sensor dimensions increase.

\paragraph{Description of the method}
Following an A-optimal design strategy, we seek to minimize the
average posterior variance of the parameter estimates, which is given
by the trace of the posterior covariance operator. For a linear
inverse problem with Gaussian prior and noise distributions, a closed
form expression for the posterior covariance operator is available and
is independent of the experimental data \cite{Tarantola05}. For
nonlinear inverse problems, however, such a closed form expression is
not available and the posterior covariance operator depends on the
experimental data. Since the data cannot be measured before the
experiment is conducted, formally this would not lead to a meaningful
OED problem. 
To cope with the dependence of the posterior covariance
$\Cpost$ on the experimental data $\obs$, we consider the average of
the trace of the posterior covariance operator over all possible
experimental data:
\begin{equation}\label{eq:intro1}
\ave_\obs \{ \trace(\Cpost(\obs) \},
\end{equation}
where $\ave_\obs$ is the expectation over data. For nonlinear inverse
problems, no closed form expressions for $\Cpost(\obs)$ are available
and the computation of $\trace(\Cpost(\obs))$ typically requires
sampling-based methods (e.g., MCMC sampling), which are particularly
expensive in high dimensions. To permit applicability to large-scale
problems, we use a Gaussian approximation of the posterior measure,
with mean given by the maximum a posteriori probability (MAP) point
$\iparmap = \iparmap(\obs)$ and covariance given by the inverse of the
Hessian operator $\H$ of the regularized data misfit functional, whose
minimizer is the MAP point. This Hessian is evaluated at the MAP
point, i.e., $\H = \H(\iparmap(\obs),\obs)$.
Notice that this approximation to the posterior is 
exact when the parameter-to-observable map is
linear. Moreover, a Gaussian is often a good approximation 
to the posterior when a nonlinear parameter-to-observable map is well 
approximated by a linearization over the 
set of parameters with significant posterior probability. 
Using this Gaussian approximation, \eqref{eq:intro1} is replaced by
\begin{equation}\label{eq:intro2}
\ave_\obs \{ \trace(\H^{-1}(\iparmap(\obs),\obs)\}.
\end{equation}
The expectation in \eqref{eq:intro2} is
approximated by averaging over a sample set
$\{\obs_1,\ldots,\obs_\Nd\}$, where each $\obs_i$ is specified
according to the noise model
\begin{equation}\label{eq:intro3}
  \obs_i=\ff(\ipar_i) + \vec\eta_i,
\end{equation}
where $\ff(\cdot)$ is the parameter-to-observable map, and $\ipar_i$
and $\vec\eta_i$ are draws from the prior and the noise distributions,
respectively.  These approximations result in a formulation of the
A-optimal design problem as a PDE-constrained optimization problem
with constraints given by the optimality conditions of the
\emph{inner} optimization problem that determines the MAP point, as
well as PDEs describing the application of the inverse of the
Hessian. 

The OED objective function involves traces of inverses of 
operators that are implicitly defined through solutions of PDEs.
We address this difficulty by using randomized trace
estimators, whose use for infinite-dimensional operators is also
addressed in this paper.
The experimental design is introduced in the Bayesian inverse problem
through a vector of non-negative weights for possible locations where
experimental data can be collected: a weight of 0 indicates absence of
a sensor, and a weight of 1 means that a sensor is placed at that
location.  To enable use of gradient-based optimization methods for an
otherwise combinatorial problem, we relax the binary assumptions on
the weights and allow them to take on any value in $[0,1]$. To control
the number of nonzero weights, and thus the number of sensors in the
experimental design, we use a sparsifying penalty
\cite{HaberHoreshTenorio08} that also favors binary weights
\cite{AlexanderianPetraStadlerEtAl14}.  Each evaluation of the OED
objective requires the solution of an inner optimization problem to find  
the MAP point (solved using an inexact Newton-CG method), and
applications of the inverse Hessian to vectors.  Gradients of the OED
objective with respect to the weights are computed efficiently using
adjoint equations, which are derived through a Lagrangian
formalism.

We elaborate the proposed OED method for the problem of inferring the log
coefficient field in an elliptic PDE. Physically this can be
interpreted as a subsurface flow problem in which we seek well
locations at which pressure data are collected so that the uncertainty
in the inferred log permeability field is minimized. 
We first consider a model problem in which we conduct a
comprehensive numerical study of the quality of the optimal
design as compared to various suboptimal 
designs. In these tests, we compare the designs by assessing their impact 
on the statistical quality of the solution of the Bayesian inverse problem. 
To this end, we compare the designs with respect to the average posterior variance as well as 
the quality of the MAP estimator which, respectively, indicate the ability of the designs to 
reduce uncertainty and to reconstruct ``truth'' log permeability fields.
These tests show that optimal designs result in
significant improvements over suboptimal designs with the same
number of sensors.  
We also examine the computational complexity, in terms of the number
of forward/adjoint PDE
solves, of the components of our method, and numerically study
its scalability. 
Finally, we compute an optimal experimental design for a larger-scale subsurface flow test problem 
with the setup and the ``truth'' log permeability field
taken from the Society of Petroleum Engineers' 10th Comparative
Solution Project (SPE10).

\section{Preliminaries}
In this section, we summarize the background material required for the
formulation and solution of OED problems for
infinite-dimensional Bayesian inverse problems.
 
\subsection{Probability measures on Hilbert spaces}\label{sec:borelmeas}
Let $\hilb$ denote an infinite-dimensional separable real Hilbert
space with inner product $\ip{\cdot\,}{\cdot}_\hilb$ and induced norm
$\|\cdot\|_{\hilb}$, and $\borel(\hilb)$ the Borel $\sigma$-algebra on
$\hilb$. A probability measure on $(\hilb, \borel(\hilb))$ is called a
Borel probability measure.
We consider a Borel probability measure $\mu$ on $\hilb$ with finite
first and second moments with mean $\bar m \in \hilb$ and covariance
operator $\C:\hilb \to \hilb$. $\C$ must be
positive, self-adjoint, and of trace-class~\cite{Prato06} and
satisfies
\begin{equation*}
\int_\hilb \|m - \bar m \|_\hilb^2 \,\mu(dm) = \trace( \C).
\end{equation*}
A Borel probability measure $\mu$ on $\hilb$ is said to be Gaussian if 
and only if for each $x \in \hilb$, the functional $u\mapsto\ip{x}{u}_\hilb
\in \R$, viewed as a real-valued random 
variable on $(\hilb, \borel(\hilb), \mu)$, is Gaussian
\cite{PratoZabczyk92,Prato06}. 
We denote by $\GM{\bar m}{\C}$ a Gaussian measure on $\hilb$ with mean
$\bar m$ and covariance
operator $\C$. 

In the present work, $\hilb = L^2(\D)$ with
the standard $L^2$-inner product $\ip{\cdot\,}{\cdot}$ and induced
norm $\|\cdot\|$,
where $\D \subset \R^d$ ($d=2,3$) is a bounded domain with
sufficiently regular boundary. 
Let $(\Omega, \Sigma, \mathsf{P})$ be a probability space and let 
$\ipar:(\Omega, \Sigma, \mathsf{P}) \to (\hilb, \borel(\hilb))$ be an $\hilb$-valued random variable 
with law $\mu$, i.e.,
$\mu(E) = \mathsf{P}(\ipar \in E), \text{ for } E \in \borel(\hilb)$. 
Notice that for each $\omega \in \Omega$, $\ipar(\cdot,\omega):\D\to\R$ is a function.
Alternatively, we may consider $\ipar$ as real-valued function defined on $\D \times \Omega$,
where for each $\vec{x} \in \D$, $\ipar(\vec{x}, \cdot)$ is
a real-valued random variable, i.e., $\ipar$ is a random field.
In this paper, we consider random fields 
that are jointly measurable on $(\hilb, \borel(\hilb)) \otimes
(\Omega, \Sigma)$ and have finite second moment.
Invoking Tonelli's theorem, the pointwise variance $\var\{\ipar(\vec{x})\}$,
$\vec{x} \in \D$, satisfies,
\begin{equation}\label{eq:avvar}
\begin{split}
\int_\D \var\{ \ipar(\vec{x}) \}\, d\vec{x} 
&= \int_\D \int_\Omega \big(\ipar(\vec{x}, \omega) - \bar\ipar(\vec{x})\big)^2 \, \mathsf{P}(d\omega)  \, d\vec{x} \\ 
&= \int_\Omega \int_\D \big(\ipar(\vec{x}, \omega) - \bar\ipar(\vec{x})\big)^2 \, d\vec{x} \, \mathsf{P}(d\omega)  \\
&= \int_\Omega \norm{ \ipar(\cdot, \omega) - \bar\ipar(\cdot) }^2 \, \mathsf{P}(d\omega)  \\
&= \int_\hilb \norm{ \ipar - \bar\ipar }^2 \, \mu(d\ipar)
= \trace(\C),
\end{split}
\end{equation}
where as before $\bar \ipar$ denotes the mean of $\ipar$.  This
shows that the trace of the covariance operator is
proportional to the average of the pointwise variance over the
physical domain $\D$---a relation that is central to our formulation
of A-optimal experimental design in an infinite-dimensional
Hilbert space.

\subsection{Bayesian inversion in an infinite-dimensional Hilbert space} \label{sec:HilbertBayes}
We consider the problem of inferring the law of the parameter $m$, modeled as an 
$\hilb$-valued random variable, from observations. Here, we describe the main 
ingredients of a Bayesian inverse problem.

\paragraph{The prior distribution law}
We use a Gaussian prior distribution law
$\priorm=\GM{\iparpr}{\Cprior}$ for the inference parameter,
where the prior mean $\iparpr$ is a
sufficiently regular element of $\hilb$ and $\Cprior:\hilb \to \hilb$
a strictly positive self-adjoint trace-class operator given by the inverse of a
differential operator.  To be precise, following
\cite{Bui-ThanhGhattasMartinEtAl13,Stuart10}, we use $\C = \A^{-2}$, where
$\A$ is a Laplacian-like operator; this choice ensures that
in two and three space dimensions, $\C$ is a trace-class operator and,
thus, the distribution is well-defined.
The measure $\priorm$ induces the Cameron-Martin space $\CM
= \ran(\Cprior^{1/2}) = \dom(\A)$ which is a dense subspace of $\hilb$
and is endowed with the inner product,
\[
   \cip{x}{y} = \ip{\A x}{\A y}, \quad x, y \in \CM.
\]
In what follows, we assume that the prior mean $\iparpr$ is an element of $\CM$.

Note that the choice of a prior that is meaningful in a function space
setting is a known challenge and an active field of research,
\cite{Stuart10,LassasSaksmanSiltanen09,DashtiHarrisStuart12,DashtiStuart15}.
Gaussian priors are a common choice for
infinite-dimensional Bayesian inverse problems.  From a practical
point of view, the use of a Gaussian prior is a modeling choice.
The prior mean describes our best guess
about the uncertain parameter, which could be obtained from existing
measurements or from other available information. The covariance operator
allows modeling of the correlation lengths and of the pointwise
variance. The choices for mean and prior might depend on the
properties that are relevant for the parameter-to-observable map. For
instance, for the subsurface flow problems considered in
sections~\ref{sec:example1} and \ref{sec:example2}, the pore-scale
rock features only influence the flow in an averaged sense. Thus,
considering smoother permeability fields that describe different
types of rocks is sufficient and an effective permeability field is
all one can hope to infer from observations.
For the prior defined above, the
Green's function of the differential operator $\A$ describes the correlation
between the parameter values at different spational points,
and so one can choose $\A$ such that it incorporates the desired
correlation information. 
We also mention the article~\cite{LindgrenRueLindstroem11}, where
a detailed study of this relation between explicitly specified %
Mat{\'e}rn-type Gaussian random fields and PDE operators
is presented.

\paragraph{The parameter-to-observable map and the data likelihood}
Next, we introduce the data likelihood, which describes the distribution of experimental 
data $\obs$ for a given parameter $\ipar \in \hilb$. Here, we consider 
finite-dimensional observations $\obs \in \R^q$, and denote by 
$\like(\obs | m)$ the likelihood probability density function (pdf).
Let $\ff: \hilb \to \R^q$ denote a \emph{parameter-to-observable map},
which is a sufficiently regular (see \cite{Stuart10})
deterministic function that maps a parameter $\ipar \in \hilb$ to an
experimental data $\obs$. In the problems we target,
an evaluation of $\ff(\ipar)$ typically requires a forward solve (typically a PDE solve) 
followed by the application of an observation operator. We consider an additive Gaussian noise
model
\begin{equation*}%
    \obs = \ff(\ipar) + \vec{\eta}, \quad \eeta \sim \GM{\vec{0}}{\ncov},
\end{equation*}
where $\ncov\in \R^{q\times q}$ is the noise covariance matrix. Note
that $\vec{\eta}$ is independent of $\ipar$ and thus $\obs | \ipar
\sim \GM{\ff(\ipar)}{\ncov}$ and the likelihood is given by
\begin{equation*} %
\like(\obs | \ipar) \propto \exp\left\{ -\frac12 \big(\ff(\ipar) - \obs\big)^T \ncov^{-1} \big(\ff(\ipar) - \obs\big)\right\}.
\end{equation*}

\paragraph{The Bayes formula in infinite dimensions}
The solution of a Bayesian inverse problem is the posterior measure,
which describes the probability law of the parameter $\ipar$
conditioned on observed data $\obs$. The relationship between the prior
measure, the data likelihood, and this posterior measure is described by
the Bayes formula, which in the infinite-dimensional Hilbert
space settings is given by \cite{Stuart10},
\[
   \frac{d\postm}{d\priorm} \propto \like(\obs | \ipar).
\]
Here, %
the left hand side is
the Radon-Nikodym derivative~\cite{Williams1991} of the posterior 
probability measure $\postm$ with respect to the prior measure $\priorm$.
See~\cite{Stuart10} for conditions on the parameter-to-observable map
$\ff$ that ensure that
the above Bayes formula holds.

\subsection{The maximum a posteriori probability (MAP) point}
\label{sec:map_point}
For a finite-dimensional inference problem, the MAP point is a
point in the parameter space at which the posterior pdf is maximized.
While this notion does not extend directly to infinite dimensions, one
can define the MAP point $\iparmap$ as the point $\ipar \in \hilb$
that maximizes the posterior probability of balls of radius $\eps$
centered at $\ipar$, as $\eps\to 0$. Analogous to the
finite-dimensional case, the MAP point can be found by minimizing the
functional $\J:\CM \to \R$  given by \cite{DashtiLawStuartEtAl13},
\[
\J(\ipar) \defeq \frac 12 \eip{\ff(\ipar) - \obs}{\ncov^{-1}(\ff(\ipar) - \obs)} +
\frac12 \cip{\ipar - \iparpr}{\ipar - \iparpr}.
\]
That is, 
\begin{equation}
\label{equ:inner-opt}
\iparmap = \argmin_{\ipar \in \CM} \J(\ipar). 
\end{equation}
The existence of solutions to the above optimization problem 
follows standard arguments~\cite{Stuart10}.
We point
out that \eqref{equ:inner-opt} is equivalent to a
deterministic inverse problem, where
inner products in the regularized data misfit functional $\J$ are weighted
according to the statistical description of the
problem, i.e., with the noise and prior covariance operators.
Note that the MAP point $\iparmap$ depends on the experimental data
$\obs$. This is a challenge in the context of OED, where data are not
available a priori.  Moreover, the solution of~\eqref{equ:inner-opt}
is not guaranteed to be unique.

\subsection{Experimental design in a Bayesian inverse problem}\label{sec:oed-basic}
Next, we define what we mean by an \emph{experimental
  design}, and describe how an experimental design enters in the
Bayesian inverse problem formulation.  We consider the problem of
optimal placement of sensors that measure experimental data.
We fix a collection of \emph{candidate sensor locations}, $\vec{x}_1,
\ldots, \vec{x}_\Ns$ in $\D$ and assign to each location a
non-negative weight $w_i$, which controls whether experimental data are
gathered at location $\vec x_i$, for $i=1,\ldots,\Ns$.
Thus, a design is fully specified by a weight
vector $\vec{w}:=(w_1,\ldots,w_\Ns) \in \R^\Ns_{\scriptscriptstyle\ge 0}$. 
Since an experimental design determines the subset of 
the set of candidate sensor locations at which data are
collected, $\vec w$ enters the Bayesian inverse problem through the data
likelihood, amounting to a weighted data likelihood:
\begin{equation} \label{equ:w-likelihood}
\like(\obs | \ipar; \vec{w}) \propto \exp\left\{ -\frac12 \big(\ff(\ipar) - \obs\big)^T 
\W^{1/2} \ncov^{-1} \W^{1/2}\big(\ff(\ipar) - \obs\big)\right\},
\end{equation}
where $\W = \diag({w_1,\ldots,w_\Ns})$.
Notice that this formulation assumes that the dimension of the data vector equals the number of
candidate sensor locations, i.e., $q = \Ns$.

Here, we consider uncorrelated observations, that is, the noise covariance
is diagonal, $\ncov = \diag(\sigma^{2}_1, \ldots, \sigma^2_\Ns)$. Thus,
\begin{equation}\label{equ:Wn}
\Wn:= \W^{1/2} \ncov^{-1} \W^{1/2} = \diag(w_1/\sigma^2_1, \ldots, w_\Ns / \sigma^2_\Ns).
\end{equation}
The solution of the Bayesian inverse problem with the weighted
likelihood \eqref{equ:w-likelihood} now additionally depends on the
design $\vec w$. For example, the MAP point (or estimator) $\iparmap$ is the
minimizer, with respect to $\ipar$, of the weighted cost functional,
\begin{equation}\label{equ:w-costJ}
\J(\ipar, \vec{w}; \obs) := \frac 12 \eip{\ff(\ipar) - \obs}{\Wn(\ff(\ipar) - \obs)} +
\frac12 \cip{\ipar - \iparpr}{\ipar - \iparpr},
\end{equation}
i.e., 
\begin{equation}
\label{equ:w-opt}
\iparmap(\vec{w}; \obs) = \argmin_{\ipar \in \CM} \J(\ipar, \vec{w}; \obs). 
\end{equation}
Other statistics of the posterior, such as the mean and the covariance
operator, also depend on $\vec w$.

In classical OED formulations~\cite{Pazman86, AtkinsonDonev92,
  Pukelsheim93, Ucinski05}, one commonly interprets the components of a
design vector $\vec{w}$ as probability masses for candidate sensor
location, i.e., $w_i \geq 0$ and $\sum w_i = 1$.  A practitioner might
place sensors at the candidate locations whose weights are large or
use the weights to decide which experiments to perform, and how often
to perform them (if experiments can be repeated)
to reduce the experimental noise level through repeated
experiments.
An alternate point of
view is to neglect the constraint $\sum w_i = 1$
and to incorporate a penalty function $P(\vec{w})$ instead, which
associates a cost to each sensor placed
\cite{AlexanderianPetraStadlerEtAl14, HaberMagnantLuceroEtAl12,
  HaberHoreshTenorio08}.
The simplest-to-interpret weight vector $\vec w$ contains
0's where no sensor is placed and 1's in locations where sensors
are placed. This leads to a binary optimization problem, which can be
challenging to solve. Thus, we relax the binary assumption on the
components of the weight vector, and allow the weights to take 
values in the interval $[0, 1]$ and enforce binary weights through properly chosen sparsifying
penalty functions, or continuation with a family of penalty functions
(see section~\ref{sec:sparsity}).

\subsection{Randomized trace estimation}\label{sec:randomized-trace-estimation}

We address
A-optimal experimental design problems, which require minimization of traces
of large dense covariance matrices that are defined implicitly
through their applications to vectors. In our OED method, we
approximate traces of covariance matrices using randomized trace
estimators.
These estimators approximate the trace of a matrix $\mat{A}
\in \R^{n \times n}$ via Monte-Carlo estimates of the form
$\trace(\mat{A}) \approx \frac{1}{\Ntr} \sum_{k = 1}^\Ntr
\ip{\vec{z}_k}{\mat{A} \vec{z}_k}_{\R^n}$, where the vectors
$\vec{z}_k$ are random $n$-vectors.
Reasonably accurate estimation of traces of high-dimensional covariance matrices
are possible with a small number of random
vectors; see e.g.,~\cite{AvronToledo11, Roosta-KhorasaniAscher13} for descriptions of 
different trace estimators and their convergence properties, 
and~\cite{AlexanderianPetraStadlerEtAl14,HaberHoreshTenorio08,HaberMagnantLuceroEtAl12}
for discussions regarding the use of randomized trace estimators for high-dimensional 
implicitly defined covariance operators.
There are several possibilities for the choice of random
vectors $\vec{z}_k$. The Hutchinson estimator~\cite{Hutchinson90} uses
random vectors with $\pm 1$ entries, each with a probability of
$1/2$. Another possibility, used in this paper, is the Gaussian trace
estimator, which uses Gaussian random vectors with independent standard
normal entries.  

In our numerical computations, we estimate traces of matrices that
are discretizations of covariance operators
defined on an infinite-dimensional Hilbert space.
Thus, we next briefly justify randomized trace estimation in infinite
dimensions. In particular, to define the 
infinite-dimensional analog of the Gaussian trace estimator, we
consider an $\hilb$-valued random variable $Z_\delta$ whose law is
given by $\mu_\delta = \GM{0}{\C_\delta}$, where $\C_\delta = (-\delta \Delta +
I)^{-2}$; here, $\Delta$ denotes the Laplacian operator with
homogeneous Neumann
boundary conditions, and $\delta$ is
a positive real number. Note that $\C_\delta$ so constructed is
positive, self-adjoint, and of trace-class on $L^2(\D)$, with $\D \subseteq \R^d$,  $d = 2, 3$.
Let $\A$ be a positive self-adjoint trace-class operator
on $\hilb$. First, note that 
\begin{equation}\label{equ:quadform}
    \ave\{{\ip{Z_\delta}{\A Z_\delta}}\} = \int_\hilb \ip{z}{\A z} \, \mu_\delta(dz)
                           = \trace(\A\C_\delta).
\end{equation}
Moreover, as shown in Appendix~\ref{apdx:trace_estimator}, %
$\trace(\A) = \lim_{\delta \to 0}  \trace(\A\C_\delta)$.
Hence, choosing
small values of $\delta$ provides reasonable estimates for
$\trace(\A)$.  Therefore, one is justified to use Monte Carlo
estimates of the form,
\[
\trace(\A) \approx \frac{1}{\Ntr} \sum_{i = 1}^{\Ntr} \ip{z_i}{\A z_i},
\]
where $z_i$ are realizations of $Z_\delta$ for a sufficiently small $\delta$ (in the finite-dimensional case, 
we can take $\delta = 0$). 

\section{A-optimal design for Bayesian \emph{linear} inverse problems}
\label{sec:linAoptimal}
The classical definition of an A-optimal design is for inverse problems
where the parameter-to-observable map $\ff$ is linear and one assumes
an additive Gaussian noise model. In this case, the
posterior covariance operator does not depend on the experimental
data.  Denoting by
$\Cpost(\vec{w})$ the covariance operator of the posterior measure
$\postm$ for a given design vector $\vec{w}$,  an A-optimal design is one that minimizes the average
posterior variance. This is equivalent to minimizing
$\trace\big(\Cpost(\vec{w})\big)$.  Denoting the linear %
parameter-to-observable map by $\mat{F}:\hilb \to \R^q$ and assuming a
Gaussian prior $\priorm=\GM{\cdot}{\Cprior}$, the posterior covariance
operator is $\Cpost(\vec{w}) = ( \mat{F}^* \Wn \mat{F} +
\Cprior^{-1})^{-1}$, with $\Wn$ as in~\eqref{equ:Wn}.
Notice that $\mat F$ is independent of the parameter $\ipar$ and
the experimental data $\obs$. Using a low rank singular value decomposition
of the prior-preconditioned parameter-to-observable map
$\mat{F}\Cprior^{1/2}$, computed \emph{once} upfront, enables
evaluating the A-optimal objective function
and its gradient without further PDE solves; see 
\cite{AlexanderianPetraStadlerEtAl14,HaberMagnantLuceroEtAl12}.

This A-optimal design approach leads to the following
optimization problem:
\[
\min_{\vec{w}  \in [0, 1]^\Ns} \trace( \Cpost(\vec{w}) ) + \upgamma P(\vec{w}),
\]
where $\upgamma P(\vec{w})$ controls the sparsity of the design
$\vec{w}$.
There are various options for choosing a sparsifying penalty function
$P(\vec{w})$. One possibility is to use $P(\vec{w}) = \sum_i w_i$,
which amounts to an $\ell^1$ penalty.  Here, we use a continuation
strategy with a sequence of penalty functions that asymptotically
approximate the $\ell^0$-``norm''; see section~\ref{sec:sparsity} and 
\cite{AlexanderianPetraStadlerEtAl14}.

\section{A-optimal design for Bayesian \emph{nonlinear}
  inverse problems}
\label{sec:oed-formulation}
In this section, we present a formulation of the A-optimal
experimental design criterion
for infinite-dimensional Bayesian nonlinear inverse problems.
To make the resulting OED problem computationally tractable, we
introduce a series of approximations, such that the formulation 
culminates in a \emph{Hessian constrained} bilevel optimization problem.

\subsection{Formulation}
For a design vector $\vec{w}$ and experimental data $\obs$, the 
Bayesian inverse problem with the weighted data likelihood
\eqref{equ:w-likelihood} is given by
\[
   \frac{d\postm}{d\priorm} \propto \like(\obs | \ipar; \vec{w}).
\]
Following an A-optimal design criterion, we seek to minimize the
average posterior variance of the inferred parameter over all
possible design vectors $\vec w$. From \eqref{eq:avvar} it follows
that the average variance is given by $\trace\left[\Cpost(\vec{w};
  \obs)\right]$, where $\Cpost$ is the covariance operator
corresponding to the posterior measure.  Note that for a fixed
experimental design vector $\vec{w}$, the result of the inference still depends on
the experimental data $\obs$. Since experimental data is, in general,
not available a priori, we average $\trace\left[\Cpost(\vec{w};
  \obs)\right]$ over the experimental data $\obs$, which, for given
$\ipar \in \hilb$, are distributed according to
$\GM{\ff(\ipar)}{\ncov}$, as specified by the data likelihood. 
Notice that this distribution of $\obs$ is conditioned on $\ipar$,
the parameter in the
Bayesian inverse problem. To address this issue, we 
rely on our prior knowledge of the parameter $\ipar$ as described by the prior measure, and
define the \emph{expected} average posterior
variance $\obj$ as follows:
\begin{equation}\label{equ:oed-objective-general}
    \obj(\vec{w}) := \ave_{\priorm}\ave_{\obs | \ipar}\left\{ \trace\left[\Cpost(\vec{w}; \obs)\right]\right\}
    \!= \!\int_\hilb\!\! \int_{\R^q} \!
       \trace\left[\Cpost(\vec{w}; \obs)\right] \, \mu_{\obs|\ipar}(d\obs) \, \priorm(d\ipar),
\end{equation}
where $\mu_{\obs|\ipar} = \GM{\ff(\ipar)}{\ncov}$.

\subsection{Gaussian approximation of the posterior measure}

If the parameter-to-observable map $\ff$ is linear, and given a
Gaussian prior distribution and an additive Gaussian noise model, the posterior is
also Gaussian, with mean and covariance given by closed form
expressions, namely the MAP point and the inverse of the Hessian of
the functional $\mathcal J$ defined in \eqref{equ:w-costJ},
respectively, \cite{Tarantola05,Stuart10}. However, if $\ff$ is
nonlinear, the posterior is not Gaussian and there
exists no closed-form expression for the posterior covariance
operator. As a consequence, one has to rely on techniques such as
Markov chain Monte Carlo sampling to compute the average posterior
variance~\cite{RobertCasella05}.
This requires a large number of statistically independent samples,
which in turn requires many evaluations of the parameter-to-observable
map $\ff$, which can make sampling computationally extremely expensive, in
particular for high-dimensional problems 
and expensive-to-evaluate parameter-to-observable maps $\ff$.
Thus, to make the problem at hand tractable, we consider a
Gaussian approximation of the posterior measure at the MAP point. That
is, given an experimental design $\vec{w}$ and a realization of the data $\obs$, we
compute the MAP point $\iparmap = \iparmap(\vec{w}; \obs)$ and define
the Gaussian approximation of $\postm$ as
\begin{equation*}%
   \postmGauss \defeq \GM{\iparmap(\vec{w}; \obs)}{\H^{-1}\big(\iparmap(\vec{w}; \obs), \vec{w}; \obs\big)},
\end{equation*}
where $\H\big(\iparmap(\vec{w}; \obs), \vec{w}; \obs\big)$ is the
Hessian  of \eqref{equ:w-costJ} (or an approximation of the Hessian,
e.g., the Gauss-Newton approximation). Note that, in general, $\H$ depends on
the design $\vec{w}$ and data $\obs$ both explicitly, and implicitly
through the MAP point.
Using this Gaussian approximation, we
proceed to define the following approximation $\objG$ for the OED objective
function $\obj$ defined in \eqref{equ:oed-objective-general}:
\begin{equation}\label{equ:oed-objective-Gaussian}
    \objG(\vec{w}) = \ave_{\priorm}\ave_{\obs | \ipar}\left\{ \trace\left[\H^{-1}\big(\iparmap(\vec{w}; \obs), \vec{w}; \obs\big) \right]\right\}.
\end{equation}

To ensure that the Gaussian approximation $\postmGauss$ is well
defined, we make the following assumption:
\begin{assumption}\label{ass1}
For every experimental data $\obs$ and every design vector $\vec w$ from the
  admissible set of designs
the inverse of the Hessian $\H^{-1}\big(\iparmap(\vec{w}; \obs),
  \vec{w};\obs\big)$ exists and is a positive trace-class operator.
\end{assumption}

\subsection{Sample averaging and randomized trace estimation}\label{subsec:sample}
The evaluation of $\objG$ given in \eqref{equ:oed-objective-Gaussian}
involves integration over an infinite-dimensional (upon
discretization, high-dimensional) space. To approximate this
integration,  we replace $\objG$ by the Monte
Carlo sum
\begin{equation}\label{eq:objGnd}
   \objG_{\!\Nd}(\vec{w}) = \frac1\Nd \sum_{i = 1}^\Nd \trace\left[\H^{-1}\big(\iparmap(\vec{w}; \obs_i), \vec{w}; \obs_i\big) \right].
\end{equation}
The data samples $\obs_i$ are given by $\obs_i = \ff(\ipar_i) + \vec{\eta}_i$, where
$\{(\ipar_i, \vec{\eta}_i)\}_{i=1}^\Nd$ is a sample set from
the product space $(\hilb, \priorm) \times (\R^q,
\GM{\vec{0}}{\ncov})$. Note that in practical computations 
usually only a moderate number of data samples can be afforded for reasons
that will become clear later in the paper.
From a frequentist's perspective, the draws
$\ipar_i$ from the prior can be considered as training models.  Note
that the draws $\obs_i$ enter in \eqref{eq:objGnd} through the
MAP point and the Hessian at the MAP point. This incorporates the
physical properties of the parameter-to-observable map $\ff$ in the
OED objective function.  For instance, if $\ff$ damps highly oscillatory 
modes of the parameters, $\objG$ is insensitive to the 
highly oscillatory modes of $\ipar_i$ used to compute $\obs_i$.  This
indirect dependence of the OED objective on ``training'' draws from the prior
is in contrast to the OED approach for nonlinear inverse problems
proposed in \cite{HaberHoreshTenorio10,HoreshHaberTenorio10}, in which
training models enter in the OED objective function directly.

The objective function \eqref{eq:objGnd} involves the trace of
$\H^{-1}_i = \H^{-1}(\iparmap(\vec{w}; \obs_i), \vec{w}; \obs_i)$. This
trace is given by $\trace[\H^{-1}_i] = \sum_{k=1}^\infty
\ip{e_k}{\H^{-1}_i e_k}$, where $\{e_k\}$ is a complete orthonormal set
in $\hilb$. Thus, we can write \eqref{eq:objGnd} as follows:
\begin{equation} \label{equ:oed-objective-mc}
   \objG_{\!\Nd}(\vec{w}) = 
   \frac1\Nd \sum_{i = 1}^\Nd \sum_{k=1}^\infty \ip{e_k}{y_{ik}},
\end{equation}
where for $i \in\{ 1, \ldots, \Nd\}$ and $ k \in \N$:
\begin{alignat*}{2}
       \iparmap(\vec{w}; \obs_i) &= \displaystyle \argmin_\ipar \J\big(\ipar, \vec{w}; \obs_i \big)&& \\%\quad i = 1, \ldots, \Nd, \\
       \H\big(\iparmap(\vec{w}; \obs_i), \vec{w}; \obs_i\big) y_{ik} &= e_k. &&
\end{alignat*}
Notice that for each $i \in \{1, \ldots, \Ns\}$, we obtain a MAP point $\iparmap(\vec{w}; \obs_i)$, which is
used to define the corresponding Hessian operator $\H_i = \H\big(\iparmap(\vec{w}; \obs_i), \vec{w}; \obs_i\big)$.

The computation of the trace based on a complete orthogonal basis as
in \eqref{equ:oed-objective-mc} is not practical.
We thus use a
randomized trace estimator (see
section~\ref{sec:randomized-trace-estimation}) to obtain an expression
that can be computed efficiently. This final approximation step
results in a computationally tractable
OED objective function, which
is used in the formulation of an A-optimal experimental design problem below.

\subsection{The resulting A-optimal experimental design problem}
The definitions and approximations discussed above %
result in the following formulation of
an A-optimal design objective function for a nonlinear Bayesian
inverse problem:
\begin{align}\label{equ:psihat}
   \hat\obj(\vec w)&:=
 \frac1{\Nd\,\Ntr} \sum_{i = 1}^\Nd \sum_{k = 1}^\Ntr \ip{z_k}{y_{ik}}, 
\end{align}
where $z_k$,  $k \in \{1,\ldots,\Ntr\}$, are random vectors as discussed in section
\ref{sec:randomized-trace-estimation}, and for $i \in\{ 1, \ldots,
\Nd\}$, $y_{ik}$ is defined through
\begin{align}
   \iparmap(\vec{w}; \obs_i) &= \displaystyle \argmin_\ipar \J\big(\ipar, \vec{w}; \obs_i \big), \nonumber\\
   \H\big(\iparmap(\vec{w}; \obs_i), \vec{w}; \obs_i\big) y_{ik} &= z_k.\nonumber
\end{align}
The corresponding A-optimal experimental design optimization problem,
with a sparsifying penalty term (as discussed in
section~\ref{sec:oed-basic}) is given by
\begin{equation}\label{equ:oed-optim-problem}\tag{$\mathcal P$}
\min_{\vec{w}\in [0,1]^{\Ns}} \hat\obj(\vec w) +  \upgamma P(\vec{w}).
\end{equation}
Since we rely on gradient-based methods to solve
\eqref{equ:oed-optim-problem}, in addition to Assumption~\ref{ass1},
we require the following assumption to hold.
\begin{assumption}
The OED objective $\hat\obj(\cdot)$ is continuously differentiable
with respect to the weight vector $\vec w$ for all $\vec{w}\in
  [0,1]^{\Ns}$.
\end{assumption}

\section{OED for coefficient field inference in an elliptic PDE}\label{sec:ellipticOED}
Next, we elaborate our approach for A-optimal design of experiments 
to the inference of the log coefficient field in an elliptic partial
differential equation, i.e., we consider the forward model,
  \begin{equation}\label{equ:poi}
    \begin{split}
      -\grad \cdot (\Exp{m} \grad u) &= f \quad \text{ in }\D, \\
                    u  &= g \quad \text{ on } \GD, \\
                    \Exp{m} \grad{u} \cdot \vec{n} &= h \quad \text{ on } \GN,
    \end{split}
  \end{equation}
where $\D \subset \R^d$ ($d=2,3$) is an open bounded domain with
sufficiently smooth boundary $\Gamma = \GD \cup \GN$, $\GD \cap \GN =
\emptyset$.  Here, $u$ is the state variable, $f\in L^2(\D)$ is a
source term, and $g\in H^{1/2}(\GD)$ and $h\in L^2(\GN)$ are Dirichlet
and Neumann boundary data, respectively.  The prior distribution for
$\ipar$ ensures that, almost surely, realizations of $\ipar$ are continuous in
$\bar{\D}$ . Hence, $\Exp{m}$ is positive and bounded, ensuring
existence of a solution of \eqref{equ:poi}.
Define the spaces,
\begin{equation*}%
    \Vg = \{ v \in H^1(\D) : \restr{v}{\GD} = g\}, \quad
    \V =  \{ v \in H^1(\D) : \restr{v}{\GD} = 0\},
\end{equation*}
where $H^1(\D)$ is the Sobolev space of functions in $L^2(\D)$ with 
square integrable derivatives. 
Then, the weak form of \eqref{equ:poi} reads as follows: Find $u \in \Vg$
such that
\[
    \ip{\Exp{m} \grad{u}}{\grad{p}} = \ip{f}{p} + \ip{h}{p}_{\GN}, \quad \forall p \in \V.
\]
In the following subsections, we specialize the OED 
problem~\eqref{equ:oed-optim-problem} for the inference of $\ipar$ in \eqref{equ:poi} from pointwise
observations of the state variable $u$. For theoretical aspects
of the Bayesian approach to estimating the
coefficient field 
in elliptic PDEs we refer to~\cite{Stuart10,DashtiStuart15}.

In
sections~\ref{sec:MAP} and \ref{sec:Hessian-mat-vec},
we derive expressions for the first and second derivatives of the
``inner'' problem, i.e., the inverse problem whose solution is the MAP point.
In Section~\ref{sec:poi-oed-formulation} we formulate the OED problem as a bilevel
optimization problem, constrained by PDEs characterizing the MAP point and PDEs defining
the action of the inverse Hessian. 
Then, in section~\ref{sec:oed-adjoint-grad}, we formulate the OED
objective resulting in the ``outer'' OED optimization problem, and
derive expressions for the gradient of the OED objective using associated
adjoint equations. A discussion of the complexity of evaluating the OED objective
and its gradient, in terms of the number of forward PDE solves, is provided in section~\ref{sec:complexity}.

\subsection{Optimality system for the MAP point}\label{sec:MAP}
We first specialize the (weighted) cost
functional~\eqref{equ:w-costJ}, whose minimizer is the MAP point,
for the problem of inferring $\ipar$ in \eqref{equ:poi} from
observations $\B u$, where $\B$ is a linear
observation operator that extracts measurements from $u$:
\begin{equation}\label{equ:poi-inner-opt}
\J(\ipar, \vec{w}; \obs) = \frac12 \eip{\B u - \obs}{\Wn(\B u - \obs)} + 
  \frac{1}{2} \cip{m - \iparpr}{m-\iparpr}.
\end{equation}
Here, for a given $\ipar$, the state variable $u$ is the solution 
to~\eqref{equ:poi}, $\iparpr$ is the prior mean of the log coefficient
field, and $\obs\in \R^q$
is a given data vector.
Note that every evaluation of the OED objective function in~\eqref{equ:oed-optim-problem} with a
given design $\vec w$ requires minimization of the PDE-constrained
data misfit cost functional in~\eqref{equ:poi-inner-opt}.  Hence, in what follows, we refer to the
minimization of~\eqref{equ:poi-inner-opt} as the \emph{inner}
optimization problem.

We use the standard variational approach to derive optimality
conditions for~\eqref{equ:poi-inner-opt} with fixed design $\vec w$.
The Lagrangian functional
$\LI: \Vg \times \CM \times \V \to \R$ is given by
\begin{equation}\label{eq:model:L}
  \LI(u,m,p):= \J(\ipar, \vec{w}; \obs) 
    + \ip{\Exp{m}\grad u}{\grad p} - \ip{f}{p} - \ip{p}{h}_{\GN}.
\end{equation}
Here, $p \in \V$ is the Lagrange multiplier and we use the superscript
$I$ to emphasize that the Lagrangian corresponds to the inner
optimization problem.
The formal Lagrange multiplier
method \cite{Troltzsch10} yields that, at a minimizer of
\eqref{equ:poi-inner-opt}, variations of the Lagrangian functional
with respect to all variables vanish, which yields
\begin{subequations} 
  \begin{align}
    \ip{\Exp{m} \grad u}{\grad \ut{p}} -
    \ip{f}{\ut{p}} - \ip{\ut{p}}{h}_{\GN} & = 0,     \label{eq:firststate}\\
    \ip{\Exp{m} \grad \ut u}{\grad p} 
    +\ip{\B^*\Wn(\B u - \obs)}{\ut{u}} &= 0,         \label{eq:firstadj}\\
    \cip{m - \iparpr}{\ut{m}}
    + \ip{\ut{m} \Exp{m}\grad u}{\grad p} &= 0,       \label{eq:firstcontrol}
  \end{align}
\end{subequations}
for all variations $(\ut{u}, \ut{m}, \ut{p}) \in \V \times \CM \times \V$.
Note that~\eqref{eq:firststate},
\eqref{eq:firstadj} and \eqref{eq:firstcontrol}
are the weak forms of the state, the adjoint and the gradient equations, respectively. 
The left hand side of \eqref{eq:firstcontrol} is the gradient
for the cost functional~\eqref{equ:poi-inner-opt}, 
provided that $u$ and $p$ are solutions to the state and adjoint
equations, respectively \cite{Troltzsch10,BorziSchulz12}.

\subsection{Hessian-vector application}\label{sec:Hessian-mat-vec}
To evaluate the OED objective function \eqref{equ:psihat}, systems of
the form $\H y = z$ have to be solved, where $\H$ is the Hessian
with respect to $\ipar$ of the regularized data misfit functional
defined in~\eqref{equ:poi-inner-opt}.  Using second variations of $\LI$ defined in
\eqref{eq:model:L} allows derivation of expressions for $\H y =
z$.  For $z\in\hilb\subset \CM'$, the solution $y\in \CM$ of $\H y =
z$ is obtained by solving a coupled system of PDEs: Find $(v, q, y)
\in \V\times \V\times \CM$ such that for all $(\ut{p}, \ut{u},
\ut{y}) \in \V \times \V \times \CM$ the following equations are satisfied:
\begin{subequations}\label{eq:incrementals}
  \begin{align}
    \ip{ \Exp{m} \grad v}{\grad \ut p} + \ip{y \Exp{m}\grad u}{\grad
      \ut p} &= 0, \label{eq:incrementals1}\\
    \ip{ \B^* \Wn \B v}{\ut u} + \ip{ y \Exp{m}\grad \ut u}{\grad p}
    + \ip{\Exp{m} \grad \ut u}{\grad q} &= 0,\label{eq:incrementals2}\\
 \ip{\ut{y} \Exp{m} \grad v}{\grad p} 
  +  \cip{y}{\ut{y}} + \ip{\ut{y} y \Exp{m} \grad u}{\grad p}
  +  \ip{ \ut{y} \Exp{m} \grad u}{\grad q} &= \ip{z}{\tilde{y}}. \label{eq:incrementals3}
  \end{align}
\end{subequations}
The equations \eqref{eq:incrementals1} and \eqref{eq:incrementals2}
are sometimes called incremental state and adjoint equations,
respectively, and the left hand side in \eqref{eq:incrementals3}
describes the application of the Hessian to a vector $y$.
In practice, $\H y = z$ is solved iteratively
using a Krylov method, which requires only the application of $\H$ to
vectors. This application can be computed by first solving
\eqref{eq:incrementals1} for $v$, then solving
\eqref{eq:incrementals2} for $q$, and then using these solutions in
\eqref{eq:incrementals3}.  Next, we provide explicit expressions for the OED problem for the
inference of the log coefficient field in \eqref{equ:poi}.

\subsection{The OED problem as a PDE-constrained optimization problem}\label{sec:poi-oed-formulation} 
Specializing the A-optimal experimental design
problem~\eqref{equ:oed-optim-problem} for the problem of inference of
the log coefficient field in \eqref{equ:poi} we obtain,
\begin{subequations}\label{equ:OEDpoi}
  \begin{align}
    &\min_{\vec{w}\in [0, 1]^\Ns} \, \frac{1}{\Nd \Ntr} \sum_{i=1}^\Nd \sum_{k = 1}^\Ntr \ip{z_k}{y_{ik}} + \upgamma P(\vec{w}) \label{equ:outeropt}
  \end{align}
where for $ i = 1, \ldots, \Nd$ and $ k = 1, \ldots, \Ntr \nonumber$
  \begin{align}
    \ip{\Exp{m_i} \grad u_i}{\grad \ut{p}} - \ip{f}{\ut{p}} - \ip{\ut{p}}{h}_{\GN}      &=0,   &&\forall \ut{p} \in \V, \label{equ:state}\\
    \ip{\Exp{m_i} \grad \ut u}{\grad p_i} +\ip{\B^*\Wn(\B u_i - \obs_i)}{\ut{u}}           &=0,   &&\forall \ut{u} \in \V,   \label{equ:adjoint}\\
    \!\cip{m_i - \iparpr}{\ut{m}} + \ip{\ut{m} \Exp{m_i}\grad u_i}{\grad p_i}              &=0,   &&\forall \ut{m} \in \!\CM,  \label{equ:grad}\\
    \ip{\B^* \Wn \B v_{ik}}{\ut u} + \ip{ y_{ik} \Exp{m_i}\grad \ut u}{\grad p_i}+\ip{\Exp{m_i}  \grad \ut u}{\grad q_{ik}} &= 0,   &&\forall \ut u \in \V,     \label{equ:incadjoint}\\
    \ip{\ut{y} \Exp{m_i} \grad v_{ik}}{\grad p_i} +  \cip{y_{ik}}{\ut{y}} + \ip{\ut{y} y_{ik} \Exp{m_i} \grad u_i}{\grad p_i} \nonumber \\
    \tab\tab+\ip{ \ut{y} \Exp{m_i} \grad u_i}{\grad q_{ik}}                                        &= \ip{z_k}{\tilde{y}},  &&\forall \ut y \in \CM,\label{equ:incgrad}\\
    \ip{\Exp{m_i} \grad v_{ik}}{\grad \ut p} + \ip{y_{ik} \Exp{m_i}\grad u_i}{\grad \ut p}   &=0,   &&\forall \ut p \in \V.    \label{equ:incstate}
    \end{align}
\end{subequations}
The PDE constraints \eqref{equ:state}--\eqref{equ:grad} are
the optimality system~\eqref{eq:firststate}--\eqref{eq:firstcontrol}
characterizing the MAP point $\ipar_i = \iparmap(\vec{w}; \obs_i)$. The equations 
\eqref{equ:incadjoint}--\eqref{equ:incstate} are the PDE constraints 
that describe $\H\big(\iparmap(\vec{w}; \obs_i), \vec{w}; \obs_i\big)
y_{ik} = z_k$ for $z_k \in \hilb$. 
Note also that compared to \eqref{eq:incrementals}, we have re-ordered the
equations~\eqref{equ:incadjoint}--\eqref{equ:incstate}. While
\eqref{eq:incrementals} follows the order in which the Hessian
application is computed in practice, the order in
\eqref{equ:incadjoint}--\eqref{equ:incstate} is such that the linear
(block-)operator on the left hand side is symmetric.
In summary,
\eqref{equ:OEDpoi} is a PDE-constrained optimization problem, where
the constraints are the first-order optimality conditions of a
PDE-constrained inverse problem, and a set of PDEs
describing the application of the inverse Hessian to vectors. 

\subsection{Evaluation and gradient computation of the OED objective}\label{sec:oed-adjoint-grad}
Evaluating the OED objective function in \eqref{equ:OEDpoi}
involves the following steps: (1)
find $(u_i, \ipar_i, p_i)$ that satisfy
equations~\eqref{equ:state}--\eqref{equ:grad} and (2) find
$(v_{ik}, y_{ik}, q_{ik})$ that
satisfy~\eqref{equ:incadjoint}--\eqref{equ:incstate}, for $i \in \{1,
\ldots, \Nd\}$ and $k \in \{1, \ldots, \Ntr\}$.

To solve the optimization problem \eqref{equ:OEDpoi}, we rely on
gradient-based optimization methods. 
Thus we need efficient methods for computing the gradient 
of $\hat \Psi$ with respect to the design vector $\vec{w}$. 
Again we follow a Lagrangian 
approach, and employ adjoint variables (i.e., Lagrange 
multipliers) to enforce the PDE 
constraints~\eqref{equ:state}--\eqref{equ:incstate} in 
the OED problem. The derivation of expressions for the gradient is rather 
involved, and it deferred to Appendix~\ref{appdx:oed-gradient}. Below, 
we simply present the final expression for the gradient, which takes the form:
\begin{equation}\label{equ:oed-grad}
    \hat\obj'(\vec{w})\! =\! \sum_{i = 1}^\Nd \ncov^{-1}(\B u_i - \obs_i) \odot
    \B\ad{p}_i - \frac{1}{\Nd\Ntr} \sum_{i = 1}^\Nd \sum_{k = 1}^\Ntr
    \ncov^{-1}\B v_{ik} \odot \B v_{ik} 
\end{equation}
where $u_i, v_{ik}$ are available from the evaluation of $\hat\obj$ as
described above, $\odot$ denotes the Hadamard product,\footnote{For
  vectors $\vec{x}$ and $\vec{y}$ in $\R^n$, the Hadamard product,
  $\vec{x} \odot \vec{y}$, is a vector in $\R^n$ with components
  $(\vec{x} \odot \vec{y})_i = x_i y_i$, $i = 1, \ldots, n$.} and,
for $i \in \{1, \ldots, \Nd\}$, the
$\ad p_i$  are obtained by solving the
following systems for the \emph{OED adjoint variables} $(\ad p_i, \ad m_i, \ad
u_i)\in \V\times \CM\times \V $:
\begin{subequations}\label{equ:outer-adj}
\begin{align}
               \!\!\ip{\B^*\Wn\B\ad p_i}{\ut u}
               \!+\! \ip{\ad{m}_i \Exp{m_i}\grad \ut{u}}{\grad p_i}
               \!+\! \ip{\Exp{m_i} \grad \ut{u}}{\grad \ad{u}_i}
               &\!=\! \ip{b^1_i}{\ut u},
   \label{equ:outer-adj4-simp}
   \\ 
               \!\!\!\ip{\ut{m}\Exp{m_i}\grad{p}_i}{\grad\ad{p}_i}
               \!+\! \cip{\ad{m}_i}{\ut{m}} \!+\! \ip{\ut m \ad{m}_i \Exp{m_i}\grad u_i}{\grad p_i}
               \!+\! \ip{\ut{m}\Exp{m_i}\grad u_i}{\grad\ad{u}_i}
               &\!=\!\ip{b_i^2}{\ut m}, 
   \label{equ:outer-adj5-simp}
   \\ 
               \!\!\ip{\Exp{m_i} \grad \ut{p}}{\grad \ad{p}_i}
               \!+\! \ip{\ad{m}_i \Exp{m_i}\grad u_i}{\grad \ut{p}}
               &\!=\!\ip{b_i^3}{\ut p},  
   \label{equ:outer-adj6-simp}
\end{align}
\end{subequations}
for all $(\ut u, \ut m, \ut p) \in \V \times \CM \times \V$, with the right hand sides given by
\begin{equation}\label{equ:rhs-nice}
\begin{aligned}
  \ip{b^1_i}{\ut u} &= \frac{1}{\Nd\Ntr}\sum_{k = 1}^\Ntr \big[2 \ip{y_{ik}\Exp{m_i}\grad\ut{u}}{\grad q_{ik}} + (y_{ik}^2 \Exp{m_i} \grad \ut u, \grad p_i)\big], 
\\
  \ip{b^2_i}{\ut m} &= 
   \frac{1}{\Nd\Ntr} \sum_{k = 1}^\Ntr \big[2 \ip{ \ut{m_i}  \Exp{m_i} \grad v}{\grad q} 
  + 2 \ip{\ut m y_{ik} \Exp{m_i} \grad u_i}{\grad q_{ik}}\\
  &\qquad\qquad\qquad+ 2 \ip{\ut m y_{ik} \Exp{m_i} \grad v_{ik}}{\grad p_i}
  + \ip{\ut m y_{ik}^2 \Exp{m_i} \grad u_i}{\grad p_i}\big], 
\\
  \ip{b^3_i}{\ut p} &= \frac{1}{\Ntr\Nd} \sum_{k = 1}^\Ntr \big[ 2\ip{y_{ik} \Exp{m_i}\grad\ut{p}}{\grad v_{ik}} + \ip{y_{ik}^2 \Exp{m_i} \grad u_i}{\grad \ut p}\big].
\end{aligned}
\end{equation}
Note that the linear operator on the left hand side of
\eqref{equ:outer-adj} coincides with the left hand side operator
in~\eqref{equ:incadjoint}--\eqref{equ:incstate}, after proper identification
of variables.
The fact that the
system for the OED adjoint variables coincides with the system
describing the Hessian of the inner optimization problem, i.e., the
Hessian of $\mathcal J$, defined in \eqref{equ:poi-inner-opt},  
with respect to $\ipar$ can be exploited in numerical computations. In particular,
if a Newton solver for the inner optimization problem is
available, the implementation can easily be adapted to perform the
computations required to evaluate the OED objective, and to compute
the OED gradient.
We summarize the steps for computing the OED objective function and its gradient
in Algorithm~\ref{alg:aopt}.
\renewcommand{\algorithmicrequire}{\textbf{Input:}}
\renewcommand{\algorithmicensure}{\textbf{Output:}}
\begin{algorithm}[ht]
\caption{Algorithm for computing $\hat\obj(\vec w)$ and its gradient
  $\hat\obj'(\vec{w})$.}
\begin{algorithmic}[1]\label{alg:aopt}
\REQUIRE Design vector $\vec{w}$, trace estimator vectors $\{z_k\}_1^\Ntr$, data samples $\{\obs_i\}_1^\Nd$
\ENSURE  $\hat\obj = \hat\obj(\vec{w})$ and $\hat\obj' = \hat\obj'(\vec{w})$
\STATE Initialize $\hat \obj = 0$ and $\hat\obj' = 0$
\FOR {$i = 1$ to $\Nd$}
\STATE \texttt{/* Evaluation of the objective function */}
\STATE Compute $\iparmap(\vec{w}; \obs_i)$ 
\hfill \COMMENT{The inner optimization (inverse) problem}
\FOR{$k = 1$ to $\Ntr$}
      \STATE Solve $\H_i y_{ik} = z_k$    
\hfill \COMMENT{$\H_i = \H\big(\iparmap(\vec{w}; \obs_i)$}
\ENDFOR
\STATE $\hat \obj \gets \hat\obj + \frac{1}{\Nd\Ntr}\sum_{k=1}^\Ntr \ip{z_k}{y_{ik}}$  
\STATE \texttt{/* Evaluation of the gradient */}
\FOR{$k = 1$ to $\Ntr$}
      \STATE Compute ${v}_{ik}$ and $q_{ik}$ \hfill\COMMENT{Equations~\eqref{equ:incstate} and~\eqref{equ:incadjoint}}
\ENDFOR
\STATE Solve $\H_i\ad{m}_i = \bar{b}_i$
\hfill\COMMENT{$\bar{b}_i$ obtained from~\eqref{equ:rhs-nice} (see Appendix~\ref{appdx:discretization} for details)}
\STATE Compute $\ad{p}_i$ \hfill\COMMENT{Equation~\eqref{equ:outer-adj6-simp}}
\STATE Compute $\hat\obj' \gets \hat\obj' + \ncov^{-1} (\B u_i - \obs_i) \odot \B\ad{p}_i - \frac{1}{\Nd\Ntr}\sum_{k = 1}^\Ntr \ncov^{-1} \B v_{ik} \odot \B v_{ik}$
\ENDFOR
\end{algorithmic}
\end{algorithm}

\subsection{Scalability of the OED solver}

\label{sec:complexity} 

Here, we provide a discussion of the computational complexity and
resulting scalability of solving the OED problem
\eqref{equ:OEDpoi}. Although this discussion is qualitative in nature,
we do provide numerical evidence of the scalability of our OED solver
in section \ref{sec:example1}. The cost of solving the OED problem in
measured in terms of the number of required forward-like PDE solves,
i.e., solves of \eqref{equ:poi}, or its adjoint or incremental
variants. We measure cost in this way to remain agnostic to the specific
governing forward PDEs and the particular PDE solver employed. These
forward-like PDE solves constitute the kernel component of the OED
optimization solver, and for any non-trivial PDE forward problem, the
PDE solves overwhelmingly dominate the overall cost; the remaining
linear algebra is negligible in comparison.  Having defined cost in
this manner, scalability then requires that the number of forward-like
PDE solves is independent of problem dimensions, which for the OED
problem \eqref{equ:OEDpoi} are the (discretized) parameter dimension
and the sensor dimension $\Ns$ (the state dimension is hidden within
the forward-like PDE solver).

To assert scalability of the OED solver, we have to argue that (1) the
evaluation of the OED objective $\hat\obj$, (2) the evaluation of the
gradient of the OED objective $\hat\obj'$, and (3) the number of OED
optimization iterations are all independent of the parameter and
sensor dimensions. To make this argument, we begin by identifying a
property of the Hessian systems that are solved at each OED
optimization iteration. These Hessian systems include those arising at
each iteration of the inner optimization problem (i.e., minimizing
$\mathcal{J}$ in \eqref{equ:poi-inner-opt}), as well as the Hessian
solves characterizing the posterior covariance in the OED objective
evaluation (\eqref{equ:incadjoint}--\eqref{equ:incstate}) and those
arising in OED gradient computation \eqref{equ:outer-adj}.  Consider
the Hessian $\H$ evaluated at the MAP point and notice that $\H$ can
be written as $\H = \HM + \Cprior^{-1}$, with $\HM$ representing the
Hessian of the first term (i.e., the data misfit term) in $\J$ defined
in~\eqref{equ:w-costJ}.  As discussed
in~\cite{FlathWilcoxAkcelikEtAl11, Bui-ThanhGhattasMartinEtAl13}, the
numerical rank $r$ of the prior-preconditioned data misfit Hessian, $\HMt
= \Cprior^{1/2} \HM \Cprior^{1/2}$, is independent of the parameter
dimension and, for many inverse problems, small. Moreover, the rank is
independent of the sensor dimension as well. This parameter/sensor
dimension-independence of $r$ reflects the fact that (1) the data are
often finite-dimensional, (2) the parameter-to-observable map is often
smoothing, and (3) the prior covariance operator is of smoothing
type. The numerical rank $r$ depends on the
parameter-to-observable map, the smoothing properties of the prior,
and the true information content of the data. The rank grows initially
with parameter and sensor dimensions until all information contained
in the data about the parameters has been resolved. Beyond this, the
rank $r$ of $\HMt$ is insensitive to further increases in parameter
and sensor dimension (e.g., through mesh refinement).

Next, we analyze the computational cost (again, measured in
forward-like PDE solves) of evaluation of the OED objective function
and its gradient as detailed in Algorithm~\ref{alg:aopt}.
We rely on inexact Newton-CG with Armijo line search to solve the
inner optimization problems in step 4 of the algorithm. The
computational cost of each Newton step is dominated by the
conjugate gradient iterations.
  Using the prior covariance as a
  preconditioner for CG, the number of CG iterations will be $\O(r)$ 
  (see \cite{CampbellIpsenKelleyEtAl96} for mesh invariance properties
  of CG for operators that are compact perturbations of the
  identity). Each CG iteration involves an application of the data
  misfit Hessian, which in turn involves a pair of incremental
  forward/adjoint PDE solves; therefore, the cost, in terms of
  forward-like PDE solves, of each inner optimization problem is
  $\O(n_\text{newton} \times 2 \times r)$, where $n_\text{newton}$ is
  the total number of Newton iterations.  Note that here we do not
  take into account the inexactness of Newton-CG. If in earlier
  iterations the Newton system is solved only approximately,
  $n_\text{newton}$ can be replaced with a smaller number.
Next, we note that for each data sample $\obs_i$, $i = 1 \ldots \Nd$,
we perform $\Ntr$ Hessian solves in steps 5--7 of the algorithm, where
we solve for $y_{ik}$, $k = 1, \ldots, \Ntr$. Thus, since we use CG to
solve these systems, it follows that the computational cost, measured
in forward-like PDE solves, of evaluating the OED objective function is
\begin{equation}\label{equ:OED-obj-cost}
    \O(\Nd \times n_\text{newton} \times 2 \times r) + \O(\Nd \times \Ntr \times 2 \times r).
\end{equation}
Note also that by the mesh invariance properties of the Newton 
method for nonlinear optimization~\cite{Deuflhard04}, $n_\text{newton}$ 
is independent of the parameter dimension. 

To compute the gradient, we need to perform the computations in step
11, which entail $2 \times \Nd \times \Ntr$ PDE solves, as well as the
Hessian solves in step 13 of the Algorithm~\ref{alg:aopt}, whose cost
is $\O(\Nd \times 2 \times r)$ PDE solves.  Thus, the cost of
evaluating the OED gradient is
\[
2 \times \Ntr \times \Nd + \O(\Nd \times 2 \times r)
\]
forward-like PDE solves. 

Observe that step 6 of Algorithm~\ref{alg:aopt} involves $\Ntr$
systems with the same Hessian operator and different right hand sides.
Thus, it is possible to further reduce the complexity of the
algorithm. For instance, precomputing a low rank approximation of the
prior-preconditioned data misfit Hessian $\HMt$ (after solving the
inner optimization problem) provides an efficient method for
applications of the inverse Hessian that is free of PDE solves
\cite{Bui-ThanhGhattasMartinEtAl13,FlathWilcoxAkcelikEtAl11}. Using
this low rank approximation of $\HMt$ allows us to remove the factor
$\Ntr$ in the second term of~\eqref{equ:OED-obj-cost}. 

The final argument to make is that the number of OED optimization
iterations is parameter/sensor dimension-independent. If one solves
the OED problem \eqref{equ:OEDpoi} using a Newton method, we would
expect this to be the case. In the example of section
\ref{sec:example1}, we employ a quasi-Newton method. It is difficult
to make a dimension-independence argument for quasi-Newton for the OED
problem; however, in that section we do observe dimension independence of
OED optimization iterations.

\subsection{Sparsity control}\label{sec:sparsity}

Here we briefly comment on the sparsity enforcing penalty method used
in the present work, which is based on the approach
in~\cite{AlexanderianPetraStadlerEtAl14}. In particular, considering
the problem~\eqref{equ:oed-optim-problem}, we first solve the problem
with $P(\vec{w}) = \vec{1}^T\vec{w}$, amounting to an $\ell^1$ penalty
to obtain the minimizer $\vec{w}^*_0$.  Subsequently, we consider a
sequence of penalty functions, $P_\eps(\vec{w})$ such that as $\eps
\to 0$, $P_\eps$ approaches the $\ell^0$ norm.  To cope with the
non-convexity of these penalty functions, we follow a continuation
strategy, i.e., we decrease $\{\eps_i\}$: For $\eps_1$, we
solve~\eqref{equ:oed-optim-problem} with penalty function $P_{\eps_1}$
and the initial guess (for the optimization algorithm) given by the
$\vec{w}^*_0$.  Subsequently, for each $i \geq 2$ the problem is
solved with $P_{\eps_i}$ as the penalty function and the initial guess
given by the solution of the proceeding optimization problem
corresponding to $\eps_{i-1}$.  The precise definition of the penalty
functions $P_\eps$ used follows
\cite{AlexanderianPetraStadlerEtAl14}. In practice, we observe that a
few continuation iterations are sufficient to attain an optimal weight
vector with a 0/1 structure.

\section{Example 1: Idealized subsurface flow}\label{sec:example1}
In this section, we study the effectiveness of our OED approach
applied to the parameter estimation problem considered in
section~\ref{sec:ellipticOED}. We interpret \eqref{equ:poi} as
subsurface flow problem and thus refer to $u$ as pressure and to 
$m$ as log permeability.

\subsection{Setup of forward problem}\label{sec:prob1_forward}
To detail the forward problem~\eqref{equ:poi},
we consider the domain $\D := (0, 1) \times (0,
1)\subset\mathbb R^2$ and no volume forcing, i.e., $f = 0$. We assume
no-outflow conditions on 
$\GN:=\{0,1\}\times (0,1)$, i.e.,
the homogeneous Neumann conditions $\Exp{\ipar} \grad u \cdot \vec{n} = 0$
on $\GN$.  The flow is driven by a pressure difference between
the top and the bottom boundary, i.e., we use $u = 1$ on 
$(0,1)\times \{1\}$ and $u = 0$ on
$(0,1)\times \{0\}$. This Dirichlet part of the boundary is denoted by
$\GD:=(0,1)\times\{0,1\}$.
In Figure~\ref{fig:forwardprob}, we show the
``truth'' permeability used in our numerical tests, the corresponding
pressure and the Darcy velocity field. 
\begin{figure}[th]\centering
\begin{tikzpicture}
  \node (1) at (0*\pos, 0*\pos){\includegraphics[height=.26\textwidth, trim=0 0 0
      0, clip=true]{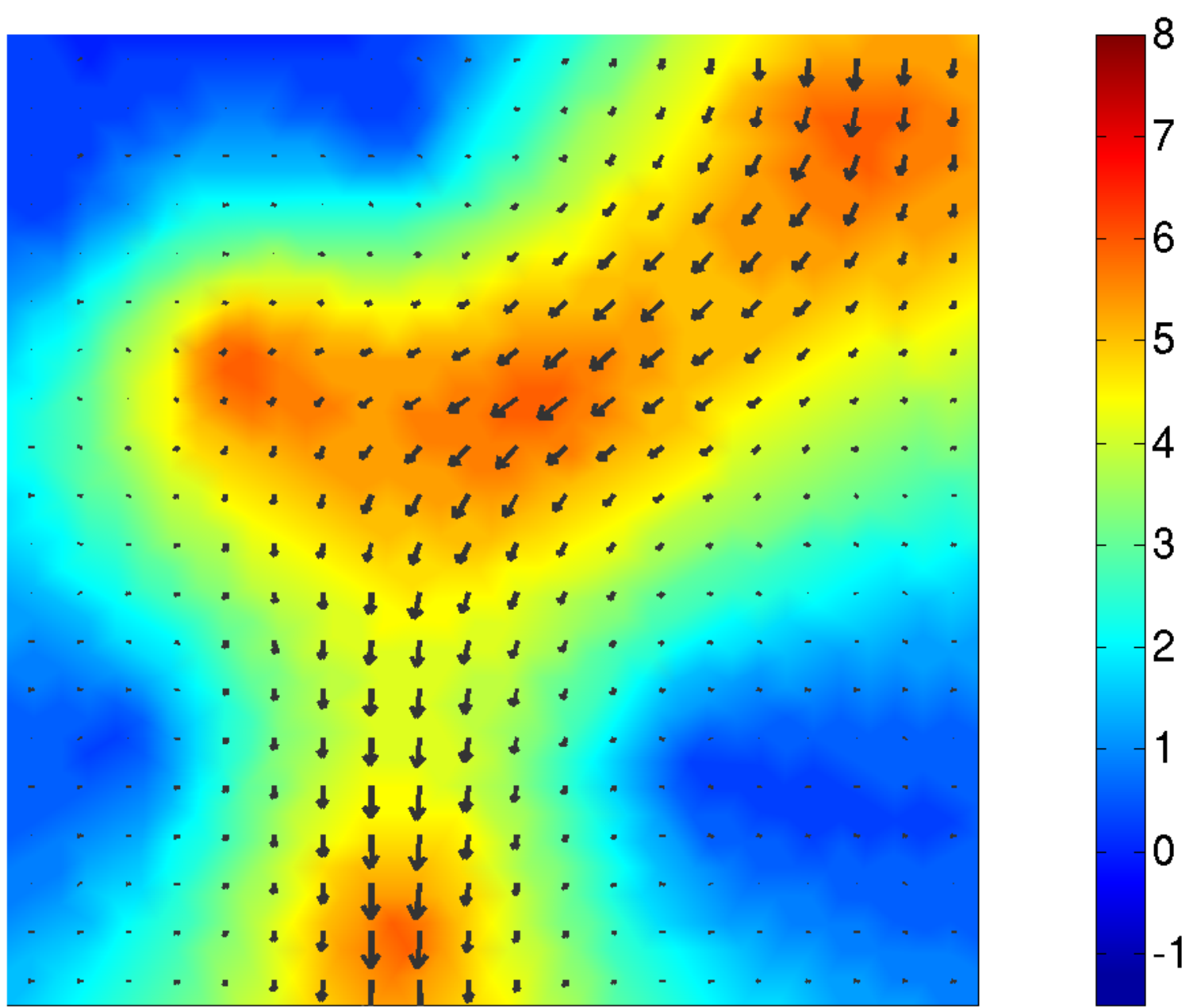}};
  \node (2) at (0.7*\pos, 0*\pos){\includegraphics[height=.26\textwidth, trim=0 0 0
      0, clip=true]{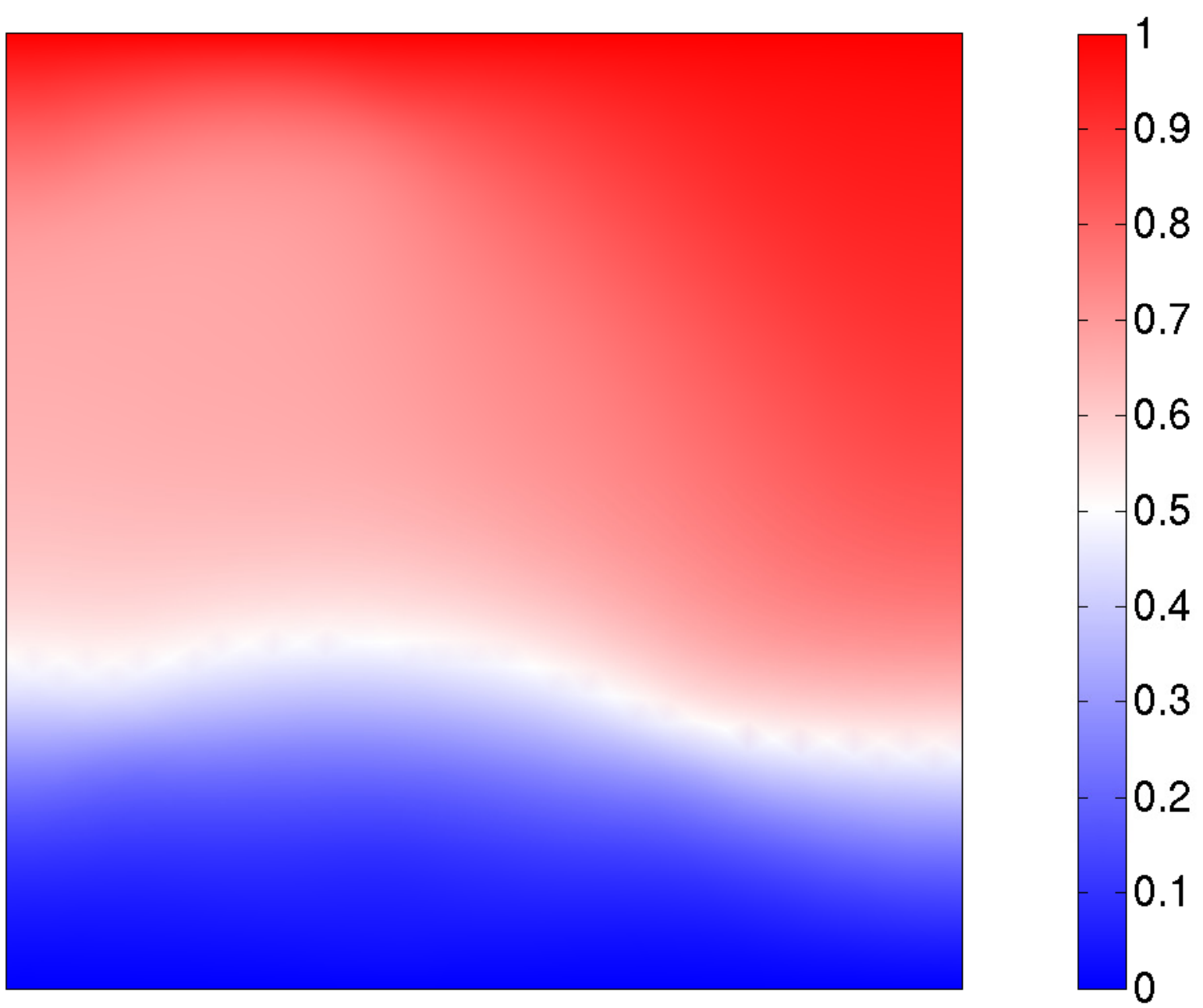}};
  \node at (0.0*\pos-0.26*\pos, -0.22*\pos) {\sf a)}; 
  \node at (0.7*\pos-0.26*\pos, -0.22*\pos) {\sf b)}; 
\end{tikzpicture}
\caption{The color in (a) shows the log permeability field $\ipart$, and
the arrows depict the Darcy velocity field, $\vec{q} = -\frac{\Exp{\ipar}}{\eta}\grad u$,
where $\ipar=\ipart$ and $\eta$ is the viscosity, here assumed to be $\eta = 1$.
The pressure field $u$ obtained by solving the state equation
with $\ipart$ is shown in (b).}
\label{fig:forwardprob}
\end{figure}

\subsection{Prior and noise model}\label{sec:prob1_prior_and_noise}
We assume given estimates $\ipart^1,\ldots,\ipart^5$ of the
log permeability at five points, i.e., $N = 5$, in $\D$, namely $\vec{x}_1 = (0.1,0.1)$,
$\vec{x}_2 = (0.1,0.9)$, $\vec{x}_3 = (0.9,0.1)$, $\vec{x}_4 =
(0.9,0.9)$, and $\vec{x}_5 = (0.5,0.5)$.
Based on this knowledge, we
compute $\iparpr$, the mean of the prior measure, as a regularized
least-squares fit of these point observations by solving
\begin{equation}\label{equ:prior_mean_prob}
    \iparpr = \argmin_{\ipar \in \CM}
\frac12 \ip{\ipar}{\A \ipar} + \frac\alpha2 \sum_{i = 1}^N 
       \int_\D 
       \delta_i(\vec{x}) \big[\ipar(\vec{x}) - \ipart(\vec{x})\big]^2 \, d\vec{x}.
\end{equation}
Here, %
$\A[\ipar]= -\grad \cdot (\mat{\Theta} \grad \ipar)$,
where the positive definite matrix $\mat\Theta$ allows to control the
prior covariance.
We define the prior covariance as $\Cprior := \mathcal{L}^{-2}$, where
$\mathcal{L} = \A + \alpha \sum_{i = 1}^N \delta_i$,
where we use the following parameter values:
\begin{equation} \label{equ:prior_parameters}
\alpha = 1, \quad \mat{\Theta} = 5\times 10^{-2}\begin{pmatrix} 1/2 & 0\\ 0 & 2\end{pmatrix}.
\end{equation}
In Figure~\ref{fig:prior}, we show the prior mean $\iparpr$, obtained by
solving \eqref{equ:prior_mean_prob} and three random draws from the
prior distribution. Note that our choice for $\Theta$ corresponds to a
prior distribution with stronger correlation in $y$-direction.
It remains to specify
the noise covariance matrix, for which we choose $\ncov = \sigma^2 I$, with
$\sigma = 0.05$.
We use a linear triangular finite element mesh with $\Nm = 1{,}121$
degrees of freedom to discretize the state, adjoint and the parameter
variables. The discrete inference parameters are the coefficients in
the finite element expansion of the parameter field.

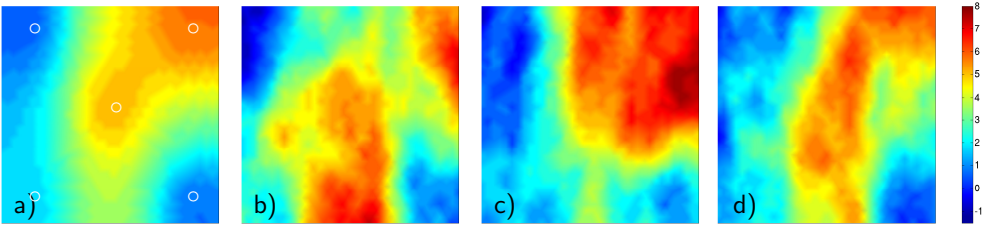
\begin{figure}[tb]\centering
\begin{tikzpicture} 
  \node (1) at (0*\pos, 0*\pos){\includegraphics[height=.23\textwidth, trim=0 0 60
      0, clip=true]{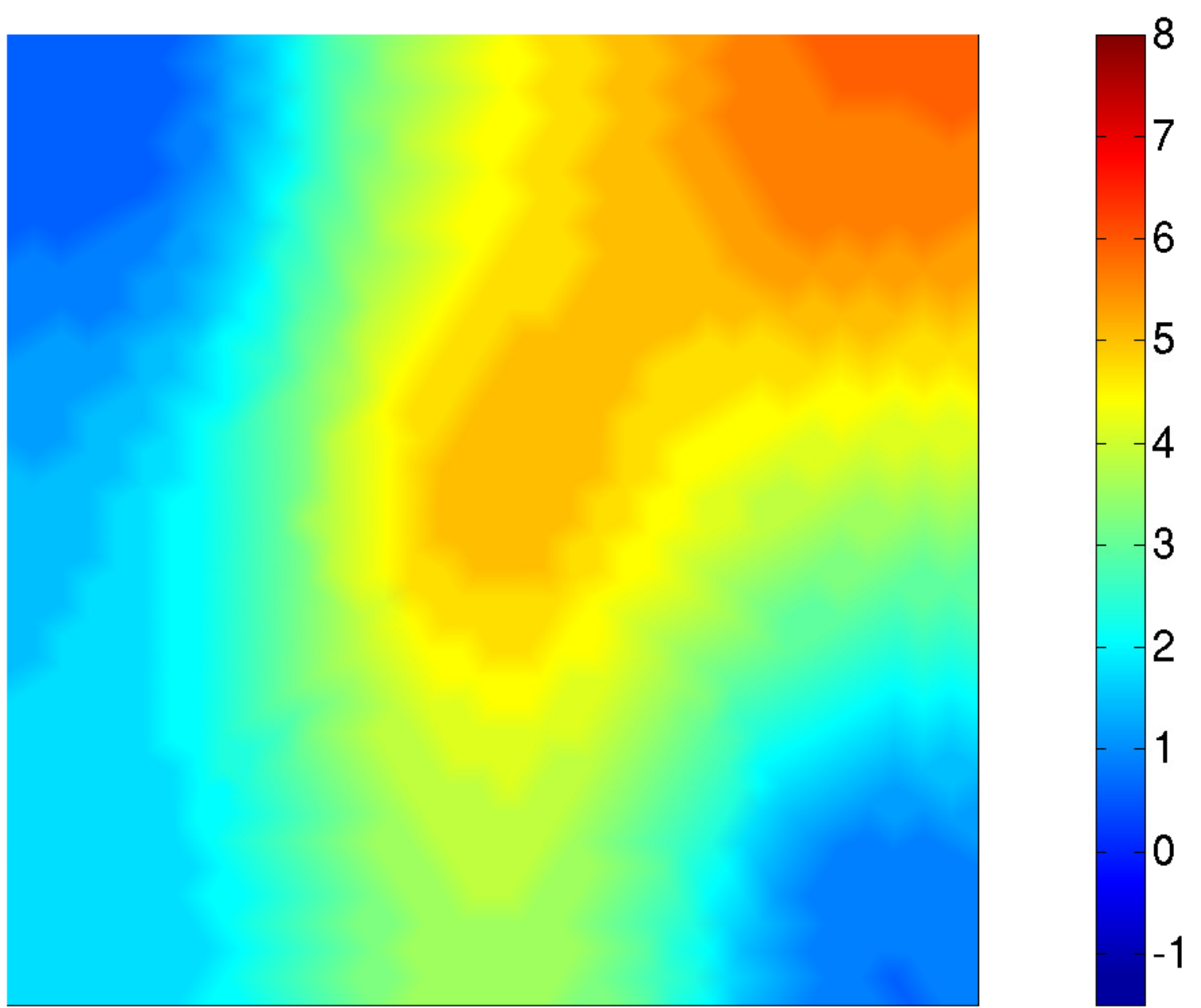}};
  \node (2) at (.5*\pos, 0*\pos){\includegraphics[height=.23\textwidth, trim=0 0 60
      0, clip=true]{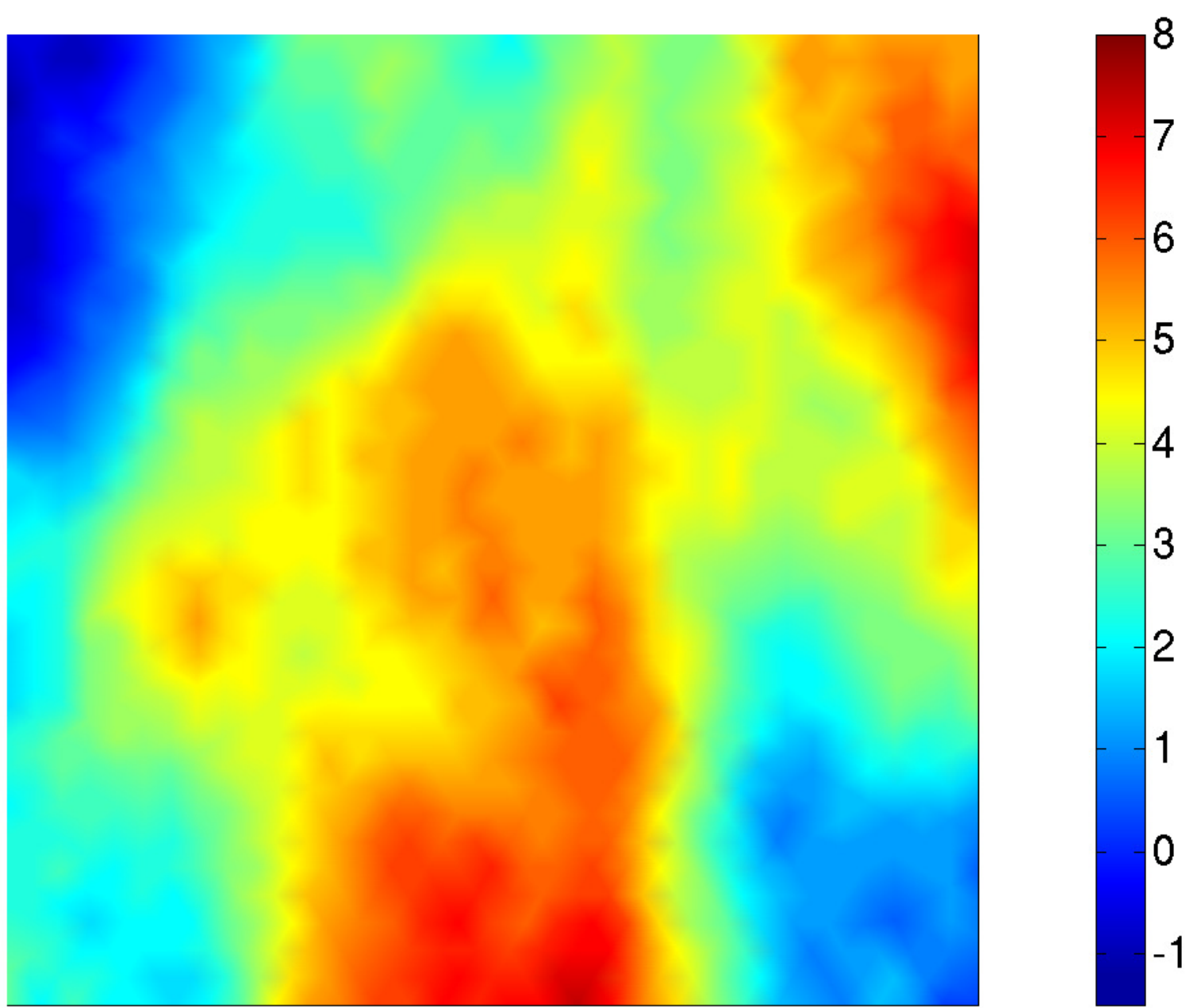}};
  \node (3) at (1*\pos, 0*\pos){\includegraphics[height=.23\textwidth, trim=0 0 60
      0, clip=true]{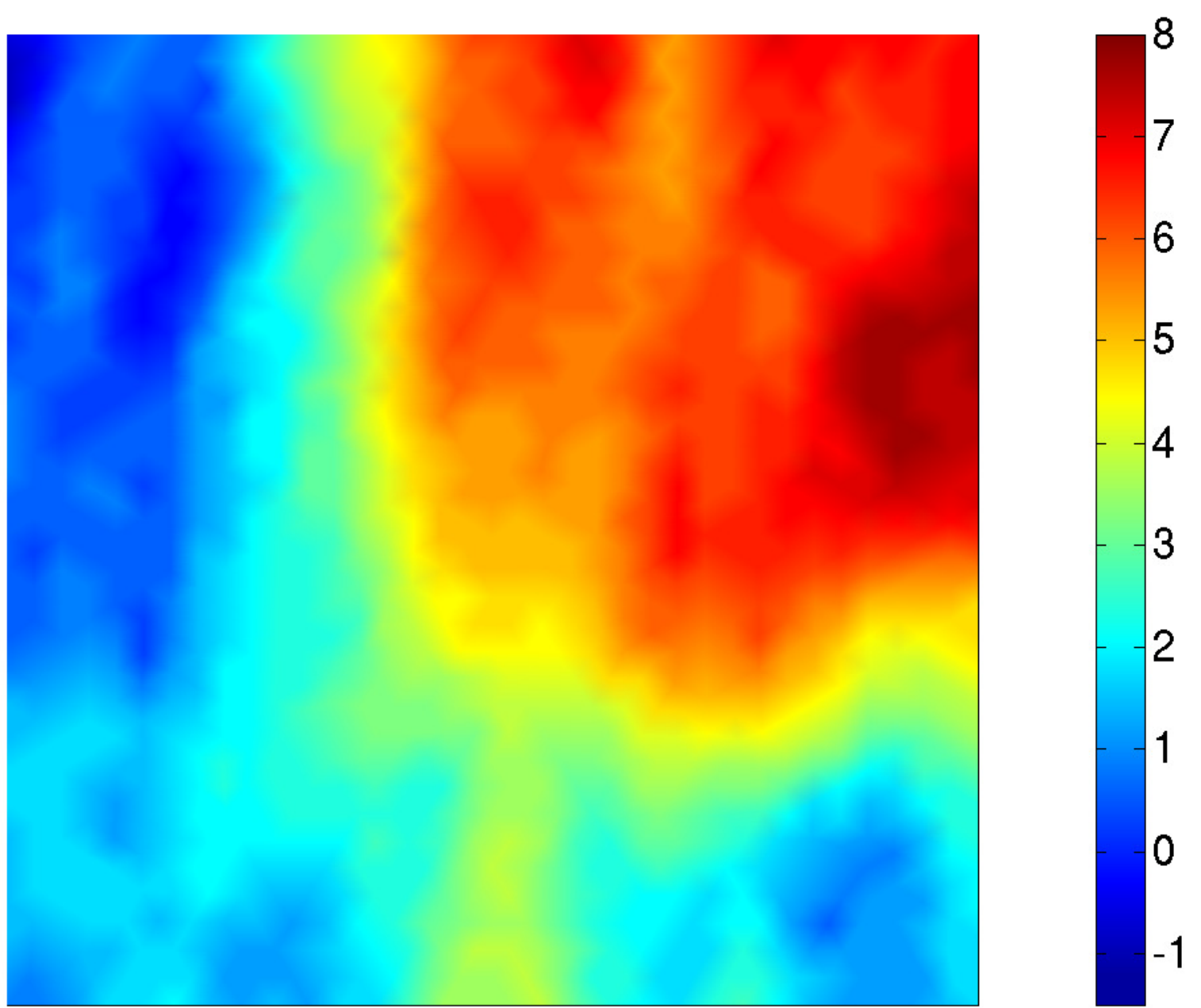}};\hspace{.2cm}
  \node (4) at (1.5*\pos, 0*\pos){\includegraphics[height=.23\textwidth, trim=0 0 0
      0, clip=true]{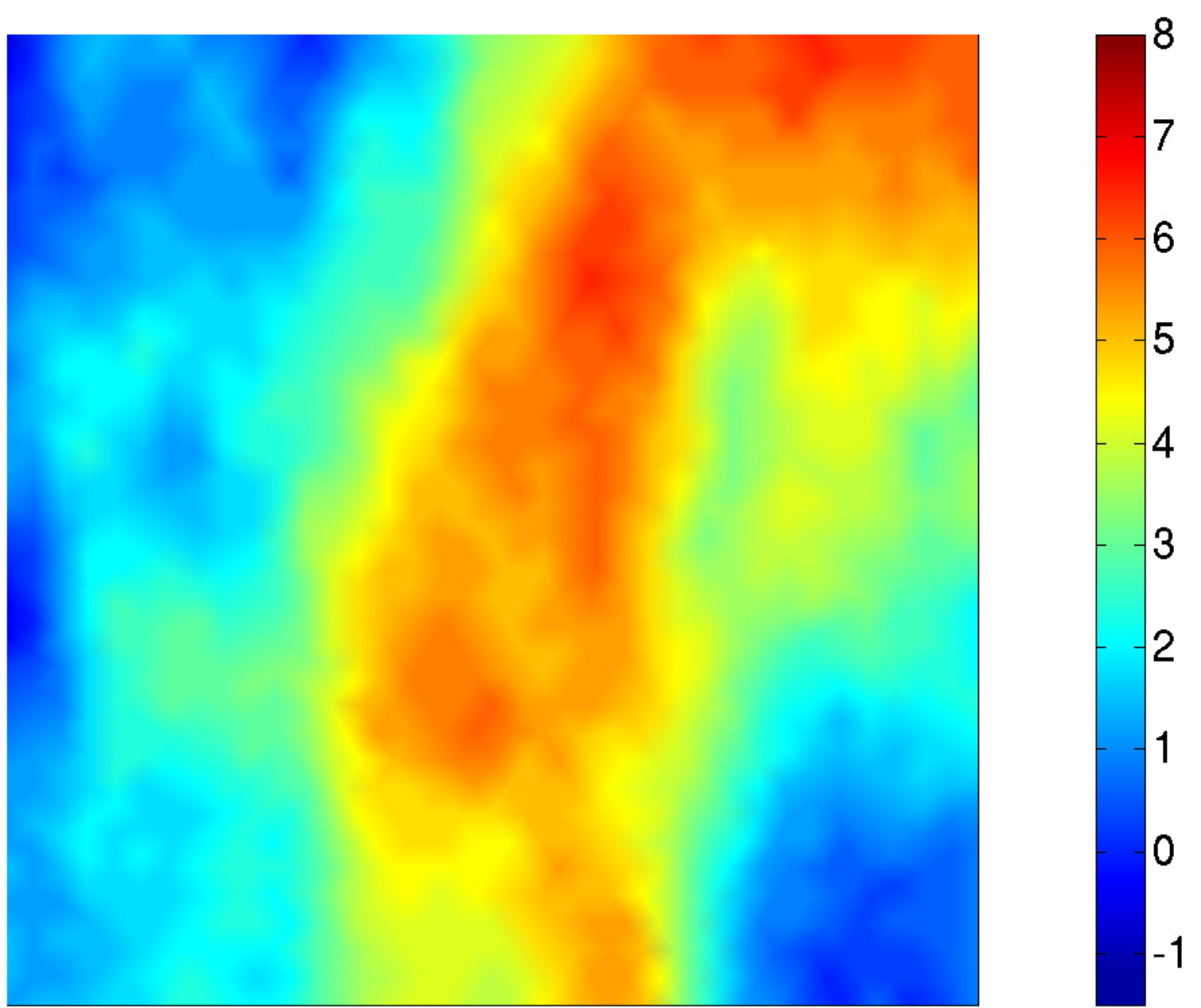}};
  \node at (0.0*\pos-0.22*\pos, -0.2*\pos) {\sf a)}; 
  \node at (.5*\pos-0.22*\pos,  -0.2*\pos) {\sf b)}; 
  \node at (1.*\pos-0.22*\pos, -0.2*\pos) {\sf c)}; 
  \node at (1.5*\pos-0.22*\pos, -0.2*\pos) {\sf d)}; 
  \node at (-0.05*\pos-0.15*\pos, -0.18*\pos) {\textcolor{white}{$\circ$}}; 
  \node at (-0.05*\pos+0.18*\pos, -0.18*\pos) {\textcolor{white}{$\circ$}}; 
  \node at (-0.05*\pos-0.15*\pos,  0.17*\pos) {\textcolor{white}{$\circ$}}; 
  \node at (-0.05*\pos+0.18*\pos,  0.17*\pos) {\textcolor{white}{$\circ$}}; 
  \node at (-0.03*\pos,           0.005*\pos) {\textcolor{white}{$\circ$}}; 
\end{tikzpicture}
\caption{Prior mean log permeability $\iparpr$ with
  circles indicating the points $\vec x_1$,\ldots,$\vec x_5$ where
  permeability measurements are available (a), and samples drawn from
  the prior distribution (b)--(d).}
\label{fig:prior}
\end{figure}

\subsection{Effectiveness of A-optimal design}\label{sec:prob1_effectivness}
We solve the OED problem~\eqref{equ:outeropt} with 
$\Nd = 5$ experimental data samples $\obs_i$, and use $\Ntr =
20$ random vectors in the trace estimator.
We employ $\ell_0$-sparsification using the continuation process
described in section~\ref{sec:sparsity}. We obtain
an optimal sensor configuration with $10$ sensors for the penalty parameter
$\upgamma = 0.008$, and an optimal design with $20$ sensors for
$\upgamma = 0.005$.
As first test of the
effectiveness of the resulting designs, we solve the inference problem 
with the ``truth'' parameter field given in Figure~\ref{fig:forwardprob}a).
Using data obtained at the A-optimal sensor configuration (with $10$ sensors), we compute
the MAP point by solving~\eqref{equ:poi-inner-opt} and the
Gaussian approximation of the posterior measure at the MAP point.  The
results are shown in Figure~\ref{fig:oed10}, where the posterior
standard deviation field is also compared with the prior standard
deviation field.
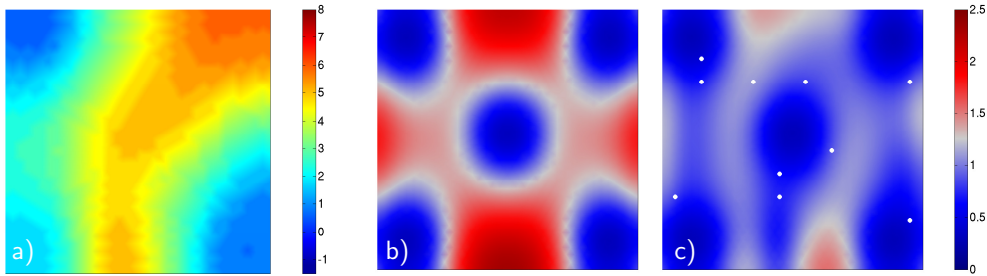
\begin{figure}[ht]\centering
  \begin{tikzpicture}
  \node (1) at (0*\pos, 0*\pos){\includegraphics[height=.28\textwidth, trim=0 0 00
      0, clip=true]{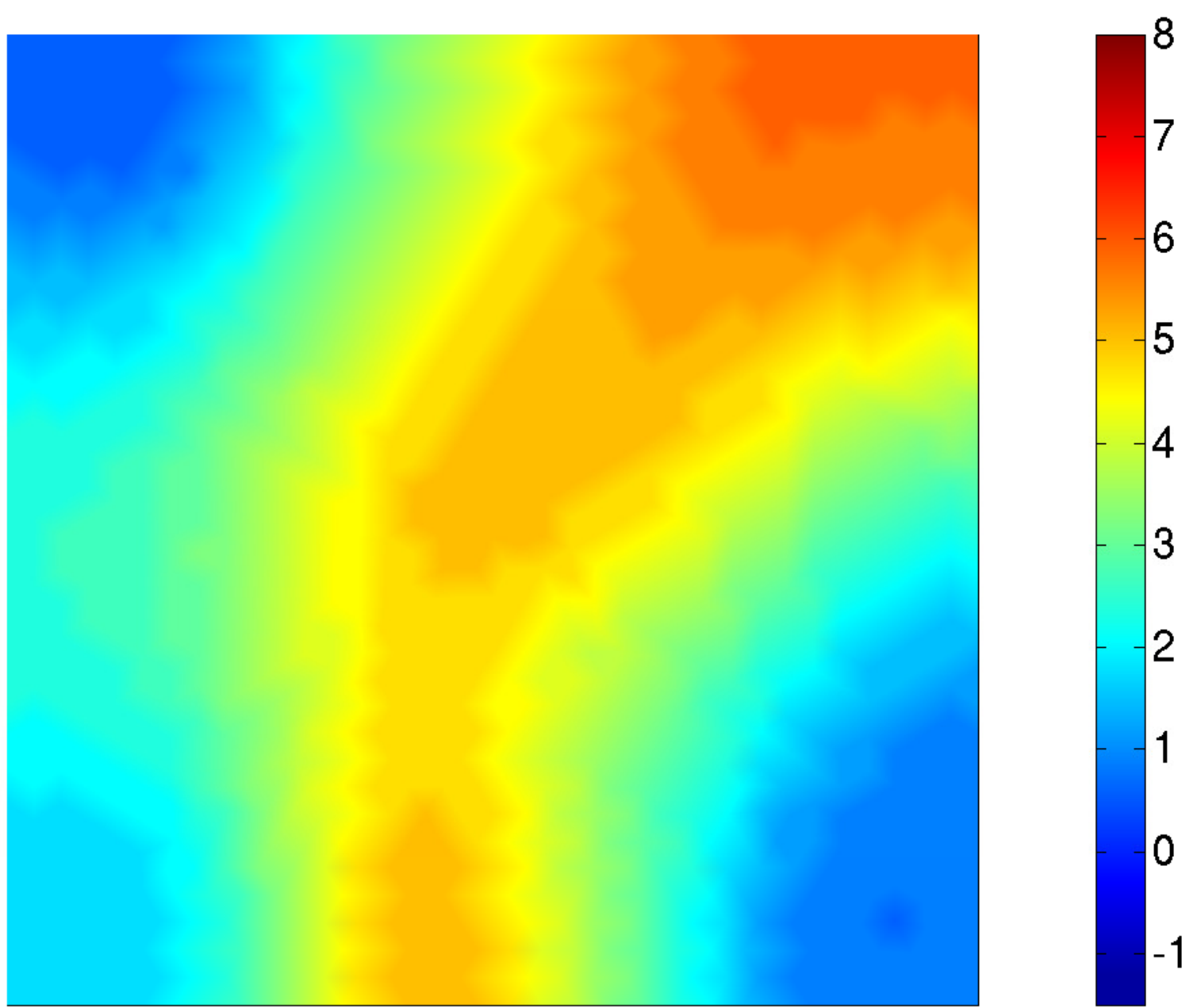}};
  \node (1) at (0.73*\pos, 0*\pos){\includegraphics[height=.28\textwidth, trim=0 0 60
      0, clip=true]{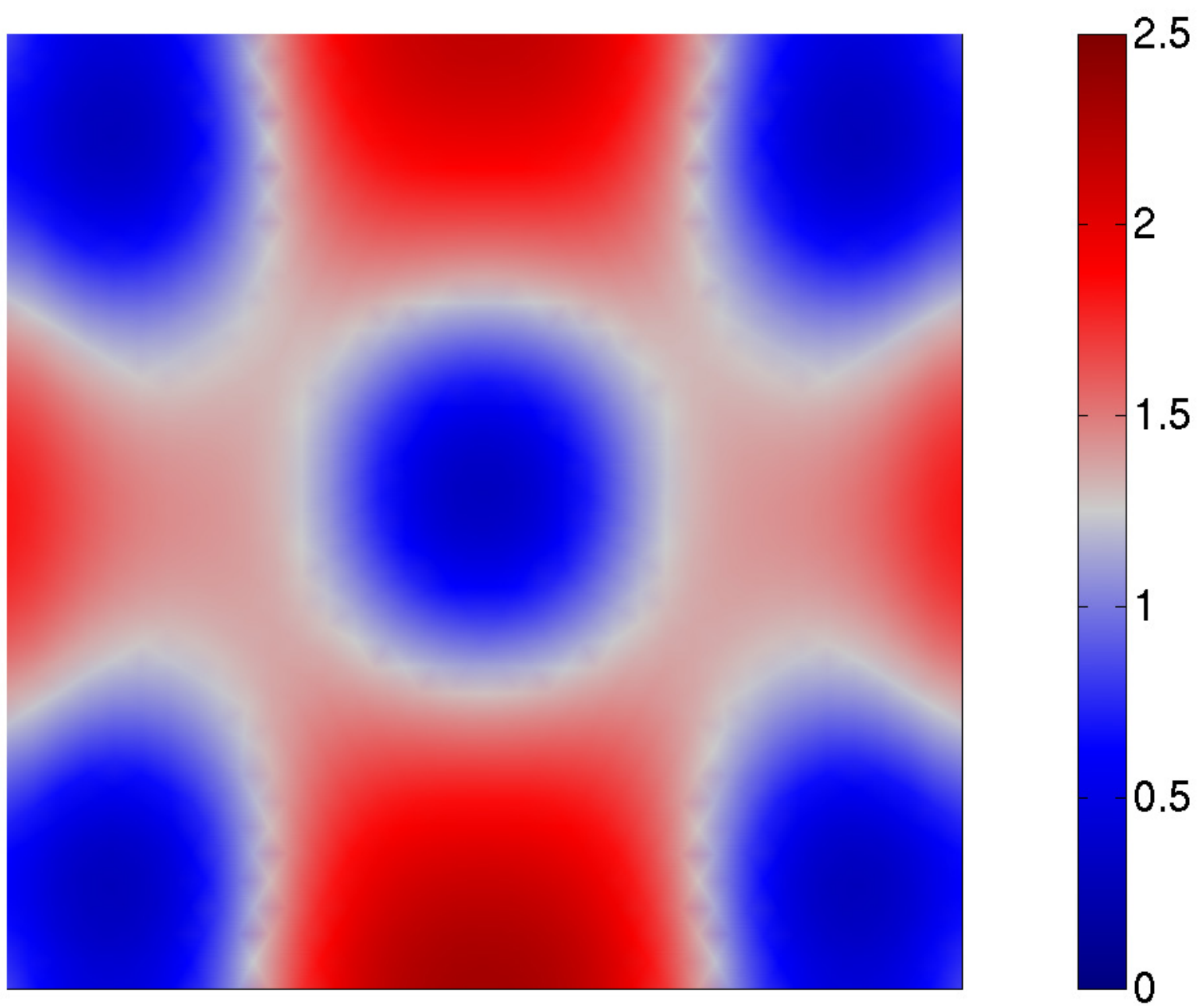}};
  \node (1) at (1.37*\pos, 0*\pos){\includegraphics[height=.28\textwidth, trim=0 0 00
      0, clip=true]{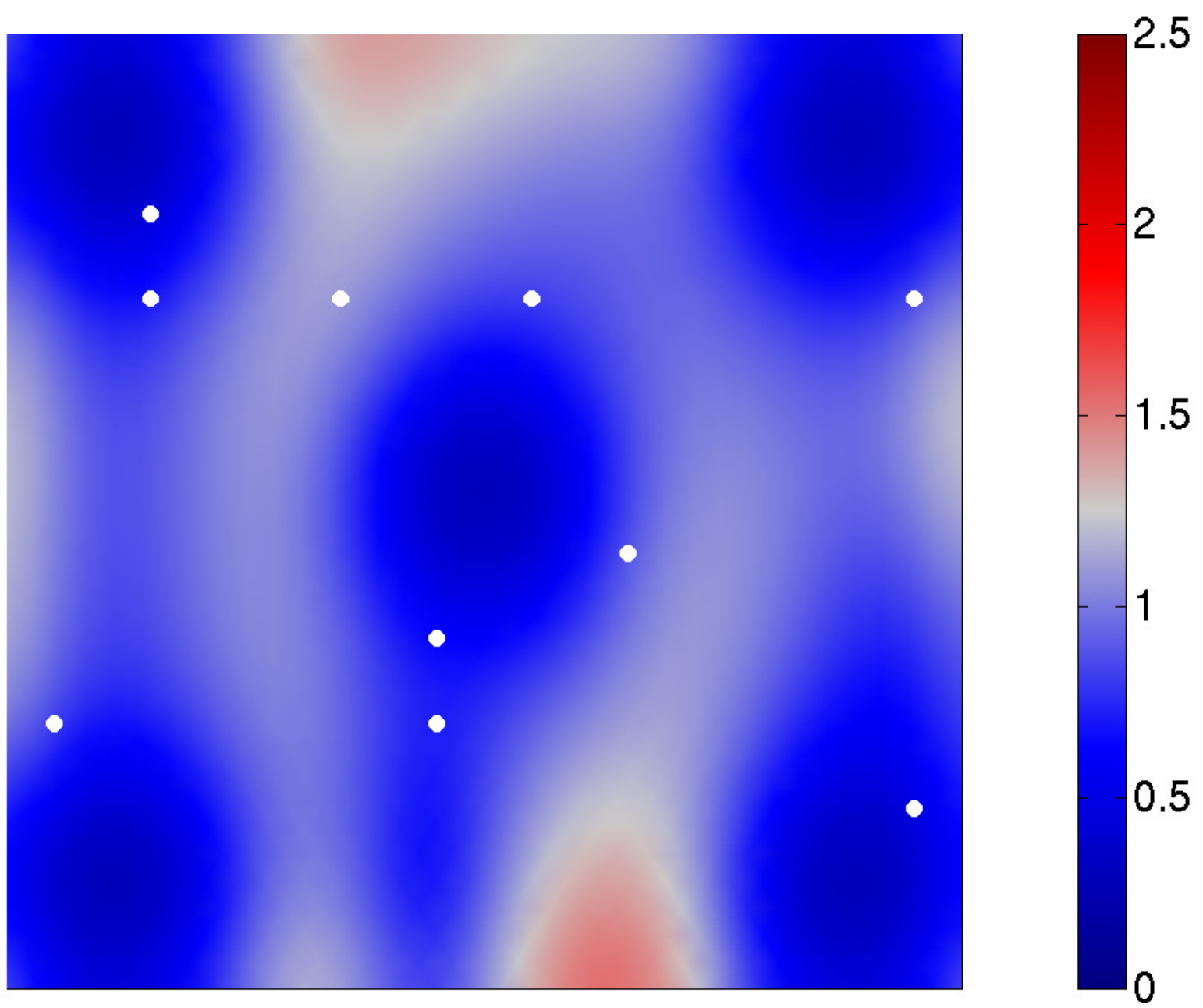}};
  \node at (0.0*\pos-0.3*\pos, -0.24*\pos) {\textcolor{white}{\sf a)}}; 
  \node at (.68*\pos-0.2*\pos,  -0.24*\pos) {\textcolor{white}{\sf b)}}; 
  \node at (1.33*\pos-0.25*\pos, -0.24*\pos) {\textcolor{white}{\sf c)}}; 
  \end{tikzpicture}
\caption{MAP point computed using the optimal design (a); prior
  standard deviation field (b) and posterior standard deviation field,
  with the optimal design sensor locations (10 sensors) indicated by
  white dots (c).}
\label{fig:oed10}
\end{figure}
To study the effectiveness of the optimal designs, we first report the
error with respect to the ``truth'' permeability field
$\ipar_\text{true}$. In Figure~\ref{fig:cloudplots} we show a
comparison of the relative error of the MAP estimator,
\[
    E_\text{rel}(\vec{w}) = \frac{\norm{\iparmap(\vec{w}) - \ipar_{\text{true}}}}{\norm{\ipar_{\text{true}}}},
\]
and of $\trace(\H(\vec{w})^{-1})=\trace(\postcov)$ for the optimal design
$\vec{w}_\text{opt}$ and for random designs with the same number of sensor
locations, where $\|\cdot\|$ is the $L^2$-norm.
From Figure~\ref{fig:oed10}, we draw the following conclusions:
(1) The optimal design with $10$ sensors improves over randomly selected designs 
more significantly than the optimal design with $20$ sensors; this indicates
that as sensors become more scarce, computing optimal design is more important.
(2) There is a correlation between minimizing the average variance and
that of minimizing the $L^2$-error of the MAP estimator.
This is interesting but not entirely surprising, because for a Bayesian 
linear inverse problem with Gaussian prior and noise, it can be shown
that minimizing the average posterior variance is equivalent to minimizing
the average mean square error of the MAP estimator~\cite{AlexanderianPetraStadlerEtAl14}.  
\begin{figure}[ht]\centering
\begin{tabular}{cc}
\begin{tikzpicture}[]
\begin{axis}[width=5cm, height=4cm, scale only axis,
    xlabel = $E_\text{rel}(\vec{w})$, ylabel=$\trace\big(\H^{-1}(\vec{w})\big)$, xmin=0.2, xmax=.62,
    ymin=0.6, ymax=1.2, legend style={font=\small,nodes=right}, legend pos= north east]
\addplot [color=red, mark=*, only marks, mark size=1pt] table[x=dist,y=tr]{./cloud10.txt};
\addlegendentry{Random}
\addplot [color=blue, mark=*, only marks, mark size=2pt] table[x=dist,y=tr]{./cloud10_opt.txt};
\addlegendentry{Optimal}
\end{axis}
\end{tikzpicture}
&
\begin{tikzpicture}[]
\begin{axis}[width=5cm, height=4cm, scale only axis,
    xlabel = $E_\text{rel}(\vec{w})$, xmin=0.2, xmax=.62,
    ymin=0.6, ymax=1.2, legend style={font=\small,nodes=right}, legend pos= north east]
\addplot [color=red, mark=*, only marks, mark size=1pt] table[x=dist,y=tr]{./cloud20.txt};
\addlegendentry{Random}
\addplot [color=blue, mark=*, only marks, mark size=2pt] table[x=dist,y=tr]{./cloud20_opt.txt};
\addlegendentry{Optimal}
\end{axis}
\end{tikzpicture}
\end{tabular}
\caption{Shown is the relative error, $E_\text{rel}$, of the MAP estimator versus
  $\trace(\H^{-1}(\iparmap(\vec{w})))$ for random designs
  $\vec{w}$ (red dots) and the optimal design $\vec{w}_\text{opt}$ (blue dot).
  These results are for designs with $10$ (left) and $20$
  (right) sensors. 
   }
\label{fig:cloudplots}
\end{figure}
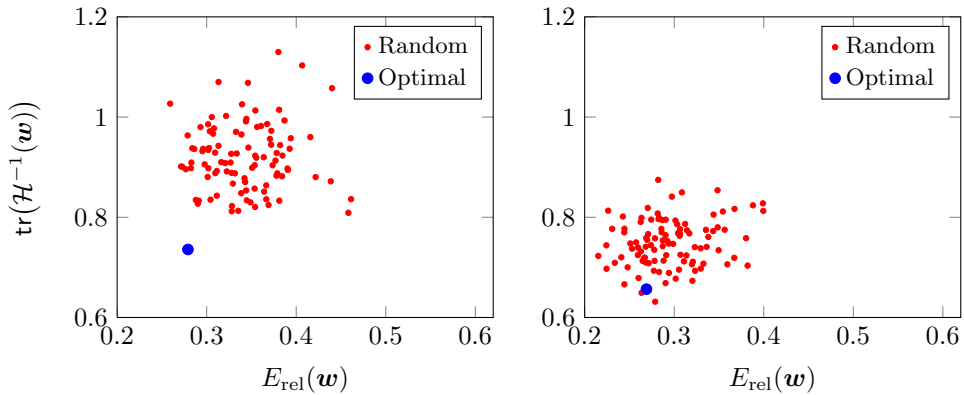
Note that the results shown in Figure~\ref{fig:cloudplots} study the
effectiveness of the OED with respect to a
specific ``truth'' model. A natural question to ask is how effective
the design is if we were trying to recover a different underlying
truth?  To address this issue, we conduct a statistical test
of the effectiveness of the optimal designs as follows. We draw
samples $\{ m_1', \ldots, m_{\Ndp}'\}$ from the prior measure and get
corresponding data vectors %
$\obs_i' = f(m_i') + \vec{\eta}_i'$, with $\eta_i'$ drawn from
$\GM{\vec{0}}{\ncov}$, $i = 1, \ldots, \Ndp$.  For a given design,
$\vec{w}$, we compute an expected error $\overline E_\text{rel}$ and
an expected average variance  $\overline{V}$:
\[
   \begin{aligned}
   \overline{V}(\vec{w}) = \frac1{\Ndp} \sum_{i = 1}^{\Ndp} \trace\big(\H^{-1}(\vec{w}, \obs_i')\big), \qquad
   \overline{E}_\text{rel}(\vec{w}) = \frac1{\Ndp} \sum_{i = 1}^{\Ndp} \frac{\norm{\iparmap(\vec{w}; \obs_i') - m_i'}}{\norm{m_i'}}.
   \end{aligned}
\]
For the
purpose of this numerical test, we let $\Ndp$ be larger than the
number $\Nd$ of the data samples used in computing the optimal
design, and the samples $\{ m_1', \ldots, m_{\Ndp}'\}$ are drawn
independently of the samples used in the sample average used for the
OED objective function (see section~\ref{subsec:sample}). Hence, $\overline V$ is
essentially a more accurate estimate of the objective function we
sought to minimize when solving the OED problem.  This allows us to
assess how well an optimal design, computed based on a small set of
data $\{\obs_1,\ldots,\obs_\Nd\}$ does in minimizing the more accurate estimate
$\overline V$.

For designs with $10$ and $20$ sensors, we compute
$\overline{V}(\vec w)$ and $\overline{E}_\text{rel}(\vec w)$ with
$\Ndp = 50$ for optimal and for $\Nw = 30$ randomly chosen designs $\vec{w}_1,
\ldots, \vec{w}_\Nw$.
The results,
shown in Figure~\ref{fig:ubercloudplots}, indicate that the
A-optimal designs computed with a relatively small number of data samples 
not only minimize the average posterior variance, but
also result in a minimal expected error between the true parameter
and the MAP point.
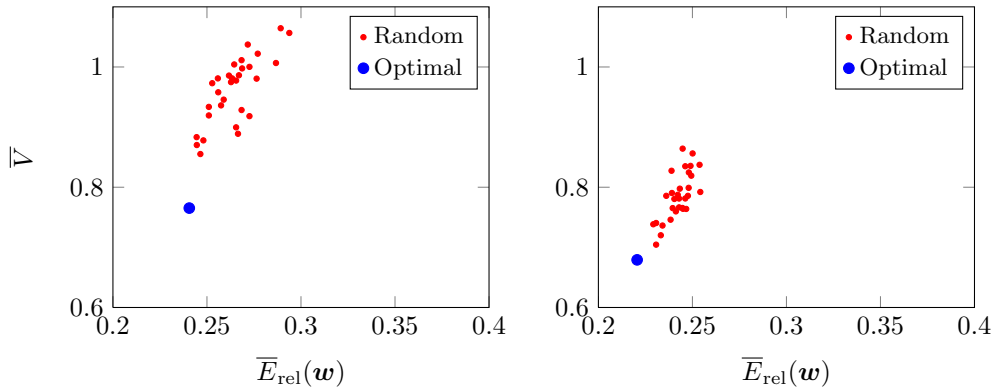
\begin{figure}[ht]\centering
\begin{tabular}{cc}
\begin{tikzpicture}[]
\begin{axis}[width=5cm, height=4cm, scale only axis,
    xlabel = $\overline{E}_\text{rel}(\vec{w})$, ylabel=$\overline V$,xmin=0.2, xmax=.4,
    ymin=0.6, ymax=1.1, legend style={font=\small,nodes=right}, legend pos= north east]
\addplot [color=red, mark=*, only marks, mark size=1pt] table[x=dist,y=tr]{./ubercloud10.txt};
\addlegendentry{Random}
\addplot [color=blue, mark=*, only marks, mark size=2pt] table[x=dist,y=tr]{./ubercloud10_opt.txt};
\addlegendentry{Optimal}
\end{axis}
\end{tikzpicture}
&
\begin{tikzpicture}[]
\begin{axis}[width=5cm, height=4cm, scale only axis,
    xlabel = $\overline{E}_\text{rel}(\vec{w})$, xmin=0.2, xmax=.4,
    ymin=0.6, ymax=1.1, legend style={font=\small,nodes=right}, legend pos= north east]
\addplot [color=red, mark=*, only marks, mark size=1pt] table[x=dist,y=tr]{./ubercloud20.txt};
\addlegendentry{Random}
\addplot [color=blue, mark=*, only marks, mark size=2pt] table[x=dist,y=tr]{./ubercloud20_opt.txt};
\addlegendentry{Optimal}
\end{axis}
\end{tikzpicture}
\end{tabular}
\caption{Expected relative error
  $\overline{E}_\text{rel}(\vec{w})$ versus expected average variance
  $\overline{V}(\vec{w})$ for random designs (red dots) and for the
  optimal design (blue dot). The left and right panels correspond to designs
  with $10$ and $20$ sensors, respectively.}
\label{fig:ubercloudplots}
\end{figure}

\subsection{Scalability and performance}\label{sec:prob1_scalability}
Finally, we examine the convergence behavior of our method as
the number of parameters and the number of sensor candidate locations
increases. Specifically, we study the computational cost in terms of
the number of solves of \eqref{equ:poi},
its adjoint, or the associated
incrementals. These elliptic PDE solves are the main building block of our 
method.

First, we consider the cost of computing the OED objective function
and its gradient. As seen in Algorithm~\ref{alg:aopt} and the
discussion in section~\ref{sec:complexity}, a
significant part of the computational cost of evaluating the OED
objective function amounts to solving the inner optimization
problem for the MAP point using an inexact Newton-CG method.
Here, the computational cost is
dominated by the CG iterations needed in each Newton step. Hence, as a
measure of the computational cost, we report the total number of
``inner'' CG iterations. We also report the number of ``outer'' CG iterations in
steps 6 and 13 of Algorithm~\ref{alg:aopt}, which are required for
computing the OED objective function and the gradient, respectively.
For this numerical study, we focused on the evaluation of the OED cost
function and its gradient at $\vec{w} = (1, 1, \cdots, 1) \in \R^\Ns$
(i.e. with all sensors active) and with $\Ntr = \Nd = 1$. The results
shown in Figure~\ref{fig:scalability} indicate that the
computational cost of evaluating the OED objective function and its
gradient are insensitive to increasing the parameter dimension, and
only depend weakly on the number of sensor candidate locations. 
Figure~\ref{fig:scalability} also shows
the number of interior point quasi-Newton iterations required for
solving the OED optimization problem, as parameter and sensor
dimensions increase.  As can be seen, the number of
iterations for solving the OED optimization problem is insensitive to
both parameter and sensor dimensions.

\begin{figure}[ht]\centering
\begin{tabular}{cc}
\begin{tikzpicture}[]
\begin{axis}[width=5cm, height=2.8cm, scale only axis, 
    xmin = 400, xmax = 6000,
    ymin = 10, ymax = 300,
    xtick = {500, 1500, 3000, 5000},
    xticklabel = \empty,
    legend style={font=\small,nodes=right}, legend pos= north east]
\addplot [color=black!40!white!60!blue, mark=*, mark size=1pt, thick, dashed] table[x=n,y=inner]{./fnEval_scalability_data_meshref.txt};
\addlegendentry{\#inner CG iter.}
\node at (axis cs:  750,  270) {\sf a)};
\addplot [color=black, mark=*, mark size=1pt, thick] table[x=n,y=outer]{./fnEval_scalability_data_meshref.txt};
\addlegendentry{\#outer CG iter.}
\end{axis}
\end{tikzpicture}
&
\begin{tikzpicture}[]
\begin{axis}[width=5cm, height=2.8cm, scale only axis, 
    xmin = 10, xmax = 800,
    ymin = 10, ymax = 300,
    xtick = {10, 100, 250, 500, 750},
    xticklabel = \empty,
    legend style={font=\small,nodes=right}, legend pos= north east]
\addplot [color=black!40!white!60!blue, mark=*, mark size=1pt, thick, dashed] table[x=ns,y=inner]{./fnEval_scalability_data_wref.txt};
\addlegendentry{\#inner CG iter.}
\node at (axis cs:  60,  270) {\sf b)};
\addplot [color=black, mark=*, mark size=1pt, thick] table[x=ns,y=outer]{./fnEval_scalability_data_wref.txt};
\addlegendentry{\#outer CG iter.}
\end{axis}
\end{tikzpicture}\\[-.3cm]
\begin{tikzpicture}[]
\begin{axis}[width=5cm, height=2.8cm, scale only axis,
    xmin = 400, xmax = 6000,
    ymin = 10, ymax = 120,
    xtick = {500, 1500, 3000, 5000},
    ytick = {10, 50, 100},
    yticklabels = {10, 50, 100}, 
    xlabel = Parameter dimension $\Nm\phantom{_s}$,
    legend style={font=\small,nodes=right}, legend pos= north east]
\addplot [color=black, mark=*, mark size=1pt, thick] table[x=n,y=iter]{./optimization_scalability_data_meshref.txt};
\addlegendentry{\#OED iter.}
\node at (axis cs:  750,  105) {\sf c)};
\end{axis}
\end{tikzpicture}
&
\begin{tikzpicture}[]
\begin{axis}[width=5cm, height=2.8cm, scale only axis,
    xmin = 1, xmax = 800,
    ymin = 10, ymax = 120,
    xtick = {10, 100, 250, 500, 750},
    ytick = {10, 50, 100},
    yticklabels = {10, 50, 100}, 
    xlabel = Sensor dimension $\Ns$,
    legend style={font=\small,nodes=right}, legend pos= north east]
\addplot [color=black, mark=*, mark size=1pt, thick] table[x=ns,y=iter]{./optimization_scalability_data_wref.txt};  
\addlegendentry{\#OED iter.}
\node at (axis cs:  60,  105) {\sf d)};
\end{axis}
\end{tikzpicture}
\end{tabular}
\caption{Computational cost of evaluating the OED objective function
  and its gradient: Shown in (a) and (b) are the total number of CG
  iterations for the inner optimization problem (to compute the MAP
  point) and for the outer optimization problem (trace estimation with inverse Hessian)
  as the parameter dimension increases (a), and as the number of
  candidate locations for sensors increases (b).  Shown in (c) and (d)
  are the number of interior-point iterations for increasing parameter dimension (c) and
  sensor dimension (d).  For (a) and (c), the sensor
  dimension is fixed at $\Ns = 100$, and for (b) and (d),
  the parameter dimension is fixed at $\Nm =
  2101$.
  }
\label{fig:scalability}
\end{figure}
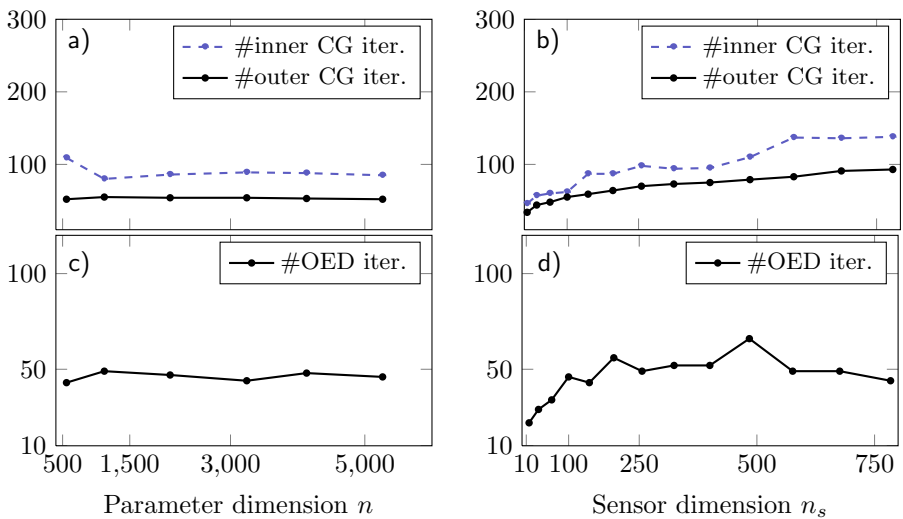

\section{Example 2: Subsurface flow based on SPE10 model}\label{sec:example2}
In this section, we consider a more realistic permeability field using
permeability field data from the Society of Petroleum Engineers' 10th
SPE Comparative Solution Project (SPE10).\footnote{See
  \url{http://www.spe.org/web/csp/datasets/set02.htm} for the
  description of the dataset.}

\subsection{Bayesian inverse problem setup}
We define the physical domain $\D = (0, 2.2) \times (0, 1.2)$ (with
unit of length in 1000's of feet) and use as the ``truth''
permeability field a vertical slice\footnote{We use the 70th slice,
  counted from the top.} of the three-dimensional SPE10 permeability
data.  Following the setup of the SPE10 model, we consider an
injection well in the center of the domain, and four production wells
at the corners of the domain. The injection well is modeled as a
mollified point source, and enters~\eqref{equ:poi} through the right
hand side function $f$ given by
$
   f(\vec{x}) = {C}/({2\pi L}) \exp \left\{ -{1}/({2L}) (\vec{x} - \vec{x}_0)^T(\vec{x} - \vec{x}_0) \right\}, 
$
with $L = 10^{-4}$ and $C = 50$, and $\vec{x}_0 = (1.1, 0.6)$. To model the production wells, we fix the pressure at zero at the four corners of the
domain. Specifically, we cut circular regions 
from the four corners of the domain (modeling the
boundaries of wells) and impose zero Dirichlet boundary conditions
on the resulting quarter circles. Homogeneous Neumann boundary conditions
are used on the remainder of the boundary.
In Figure~\ref{fig:spe_forward}, we show the ``truth'' log
permeability field, as well as the Darcy velocity field and the
pressure obtained by solving the state equation with the true
permeability field.
\begin{figure}[ht]\centering
  \begin{tikzpicture}
  \node (1) at (0*\pos, 0*\pos){\includegraphics[width=.46\textwidth, trim=0 0 00
      0, clip=true]{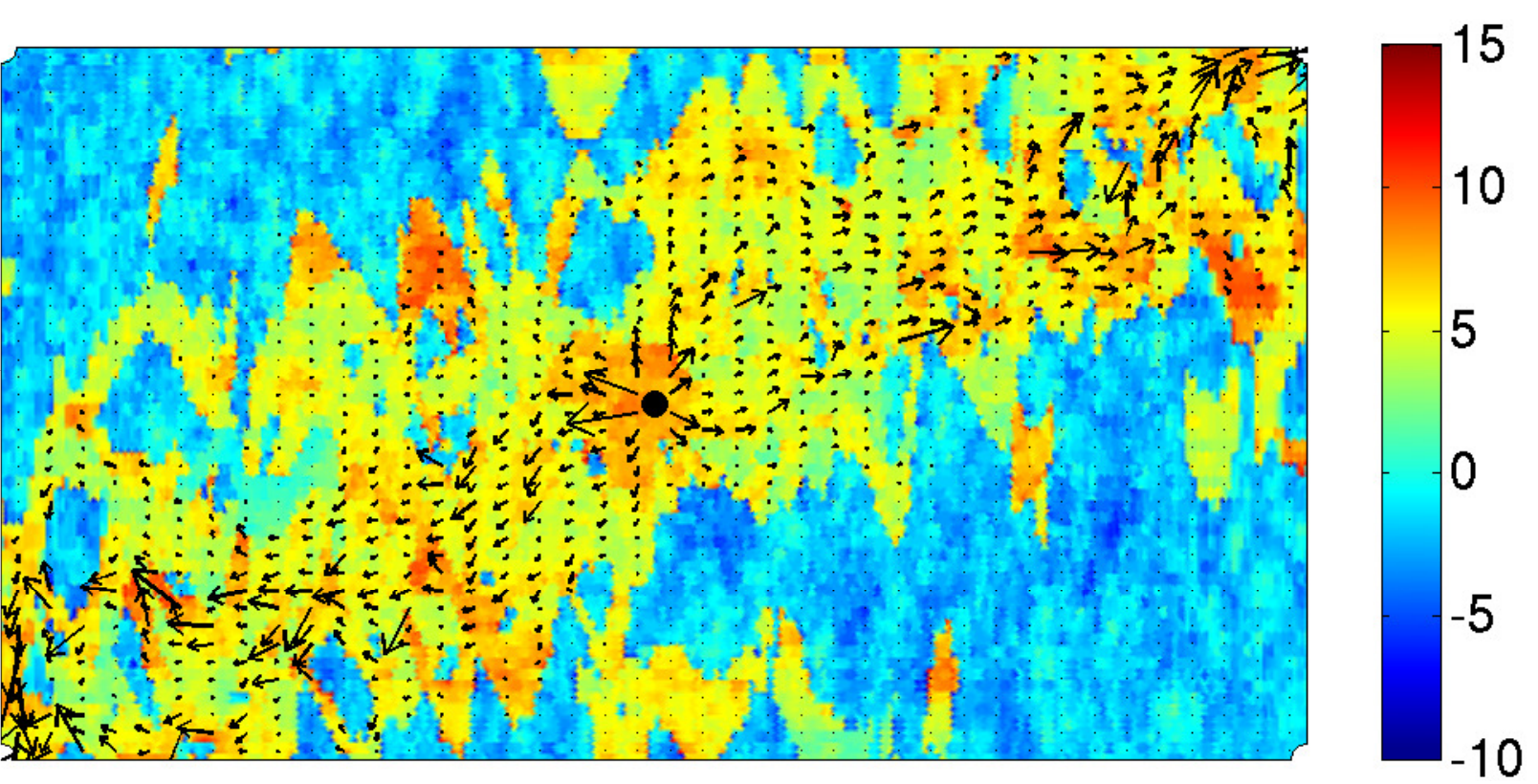}};
  \node (2) at (1*\pos, 0*\pos){\includegraphics[width=.46\textwidth, trim=0 0 00
      0, clip=true]{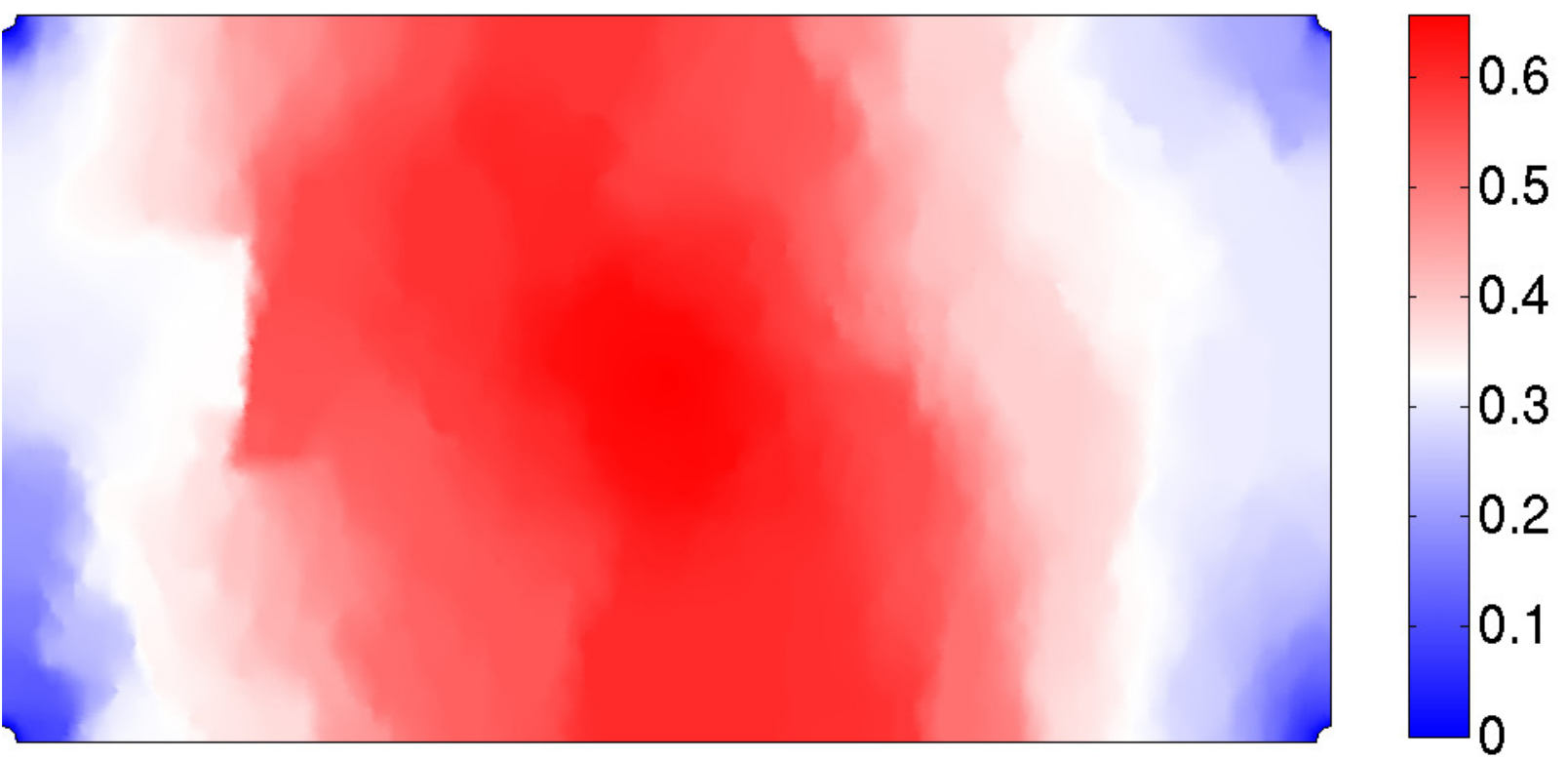}};
  \node at (0.0*\pos-0.42*\pos, 0.17*\pos) {\sf a)}; 
  \node at (1*\pos-0.42*\pos, 0.17*\pos) {\sf b)}; 
  \end{tikzpicture}
\caption{In (a) we show the ``truth'' SPE10 log permeability field
  with arrows depicting the corresponding Darcy velocity field,
  $\vec{q} = -\frac{\Exp{m}}{\eta}\grad u$, where $m=\ipart$ denotes
  the true permeability and $\eta$ is the viscosity, chosen as $\eta =
  1$.  The black dot in the center indicates the location of the
  injection well.  Shown in (b) is the pressure $u$ obtained by
  solving the state equation with $\ipart$.}
\label{fig:spe_forward}
\end{figure}

The prior construction is similar as in the previous test problem.  We
assume estimates $\ipart^1,\ldots,\ipart^5$ of the log
permeability at $N=5$ points, one at the injection well in the 
center of $\D$, and the others are near each of the four corners of the domain (at the production well boundaries).
Based on this data, we compute the
mean of the prior measure, as a regularized least-squares fit of these
point observations as in~\eqref{equ:prior_mean_prob}; see Figure~\ref{fig:spe_prior}(a).  As before, 
the prior covariance is $\C_0 = \mathcal{L}^{-2}$ where
$\mathcal{L} = -\theta\Delta + \alpha \sum_{i = 1}^N \delta_i$, with
parameter values in~\eqref{equ:prior_parameters} given by $\theta =
3.54\times10^{-2}$ and $\alpha = 1.25\times10^1$.
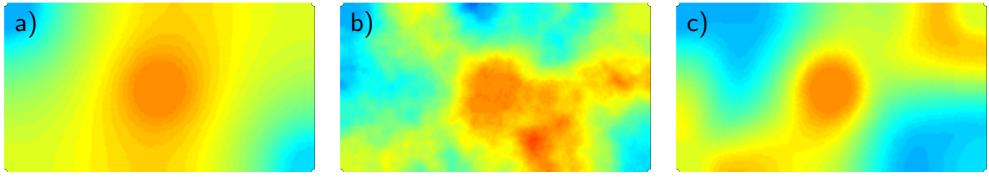
\begin{figure}[tb]\centering
\begin{tikzpicture}
  \node (1) at (0*\pos, 0*\pos){\includegraphics[width=.32\textwidth, trim=0 0 00
      0, clip=true]{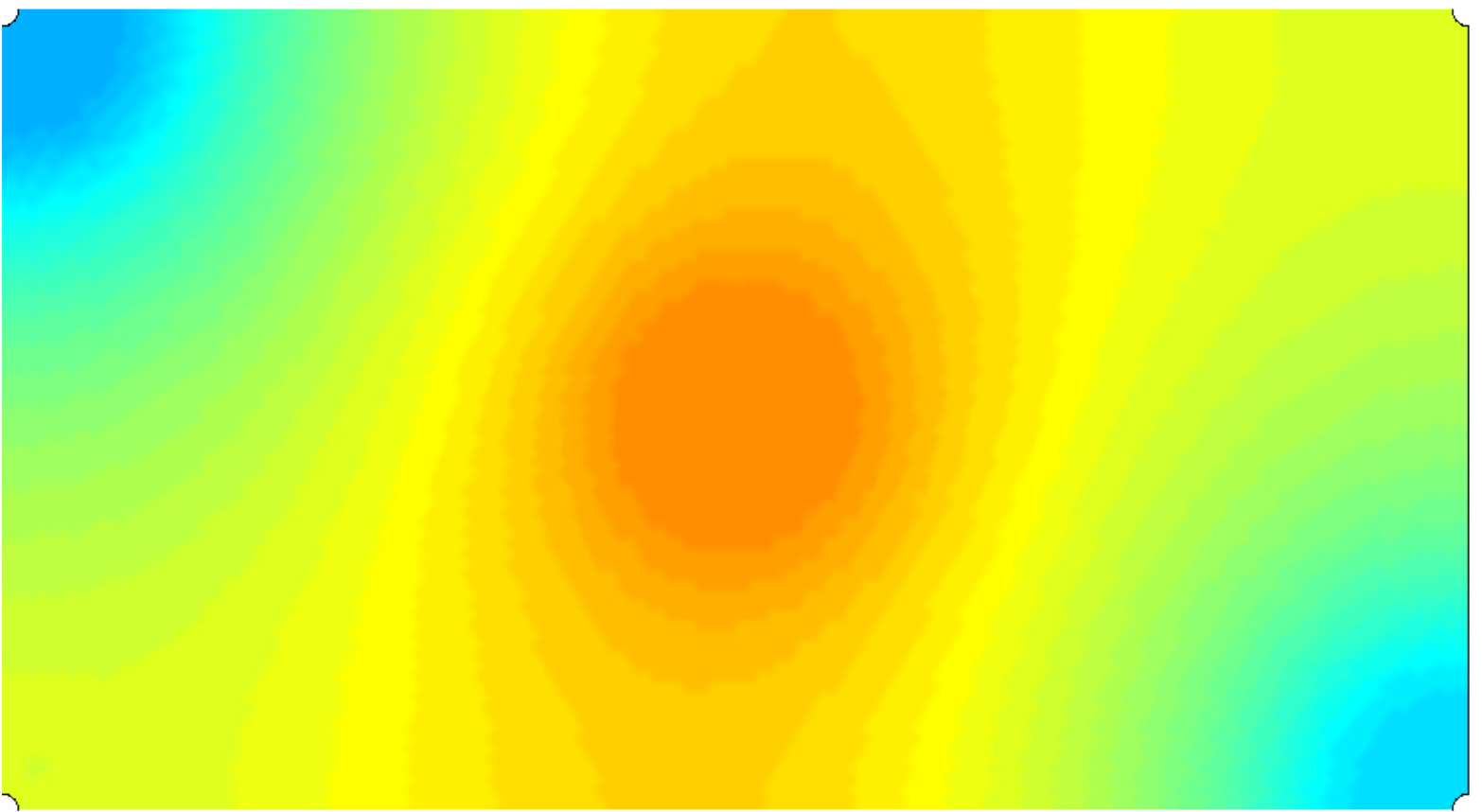}};
  \node (2) at (0.7*\pos, 0*\pos){\includegraphics[width=.32\textwidth, trim=0 0 00
      0, clip=true]{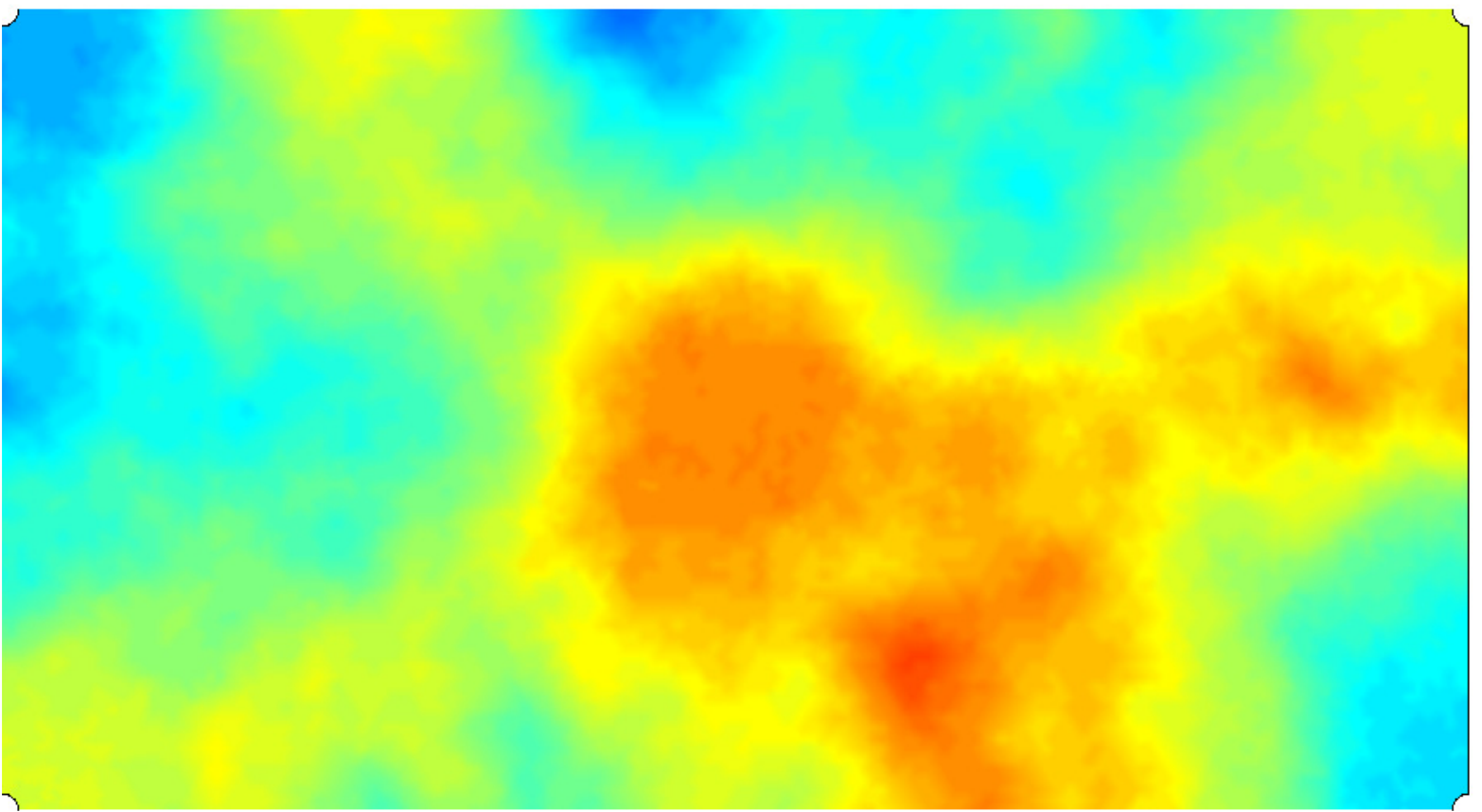}};
  \node (3) at (1.4*\pos, 0*\pos){\includegraphics[width=.32\textwidth, trim=0 0 00
      0, clip=true]{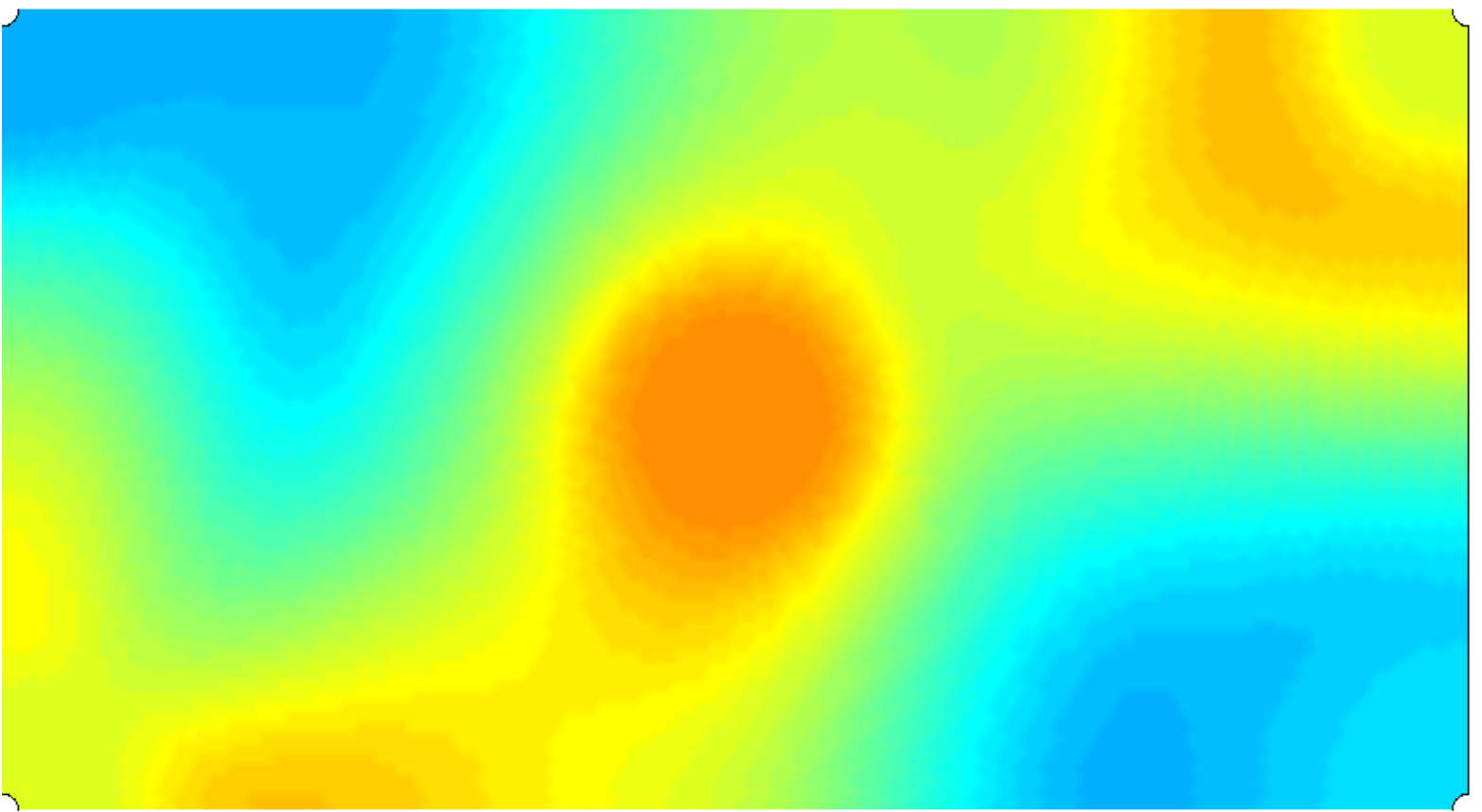}};
  \node at (0.2*\pos-0.48*\pos, 0.13*\pos) {\sf a)}; 
  \node at (0.9*\pos-0.48*\pos, 0.13*\pos) {\sf b)}; 
  \node at (1.6*\pos-0.48*\pos, 0.13*\pos) {\sf c)}; 
\end{tikzpicture}%
\caption{Prior mean (a), sample draw from the prior used to generate
  data for the OED process (b), and, in (c) the MAP point found with on the optimal
  sensor locations shown in Figure~\ref{fig:spe_aopt}b.}
\label{fig:spe_prior}
\end{figure}
Linear triangular finite elements with $\Nm = 10{,}202$
degrees of freedom are used to discretize the state, adjoint and the parameter
variables.

\subsection{A-optimal design of experiments}
We use a grid of $128$ candidate 
sensor locations in the domain $\D$, and compute an A-optimal design based on one data sample, 
computed using one random draw from prior depicted in
Figure~\ref{fig:spe_prior}(b). 
For the OED objective function, given in~\eqref{equ:outeropt}, we use a trace estimator with $\Ntr = 20$ random vectors. 
After six continuation iterations, our method converged to a 0/1
design vector. In each continuation step we terminated the interior-point iterations if either the relative residual 
fell below $10^{-5}$ or if we reached a maximum of 100 interior-point BFGS iterations.

We solve the Bayesian inverse problem using experimental data at the
A-optimal sensor locations for the ``truth'' log permeability
$\ipart$. To capture the extreme variations in the
permeability field, we solve the forward problem using quadratic
triangular elements on a finer mesh with $n = 237,573$ degrees of
freedom, and record pressure measurements at the sensor
sites. This data vector is subsequently used in the solution of
the Bayesian inverse problem.  After solving the Bayesian inverse
problem with the A-optimal sensor configuration, in Figure~\ref{fig:spe_prior}c, 
we show the MAP point, and in
Figure~\ref{fig:spe_aopt}, compare the prior and posterior standard deviation
fields. %
\begin{figure}[ht]\centering
  \begin{tikzpicture}
  \node (1) at (0*\pos, 0*\pos){\includegraphics[height=.25\textwidth, trim=0 0 00
      0, clip=true]{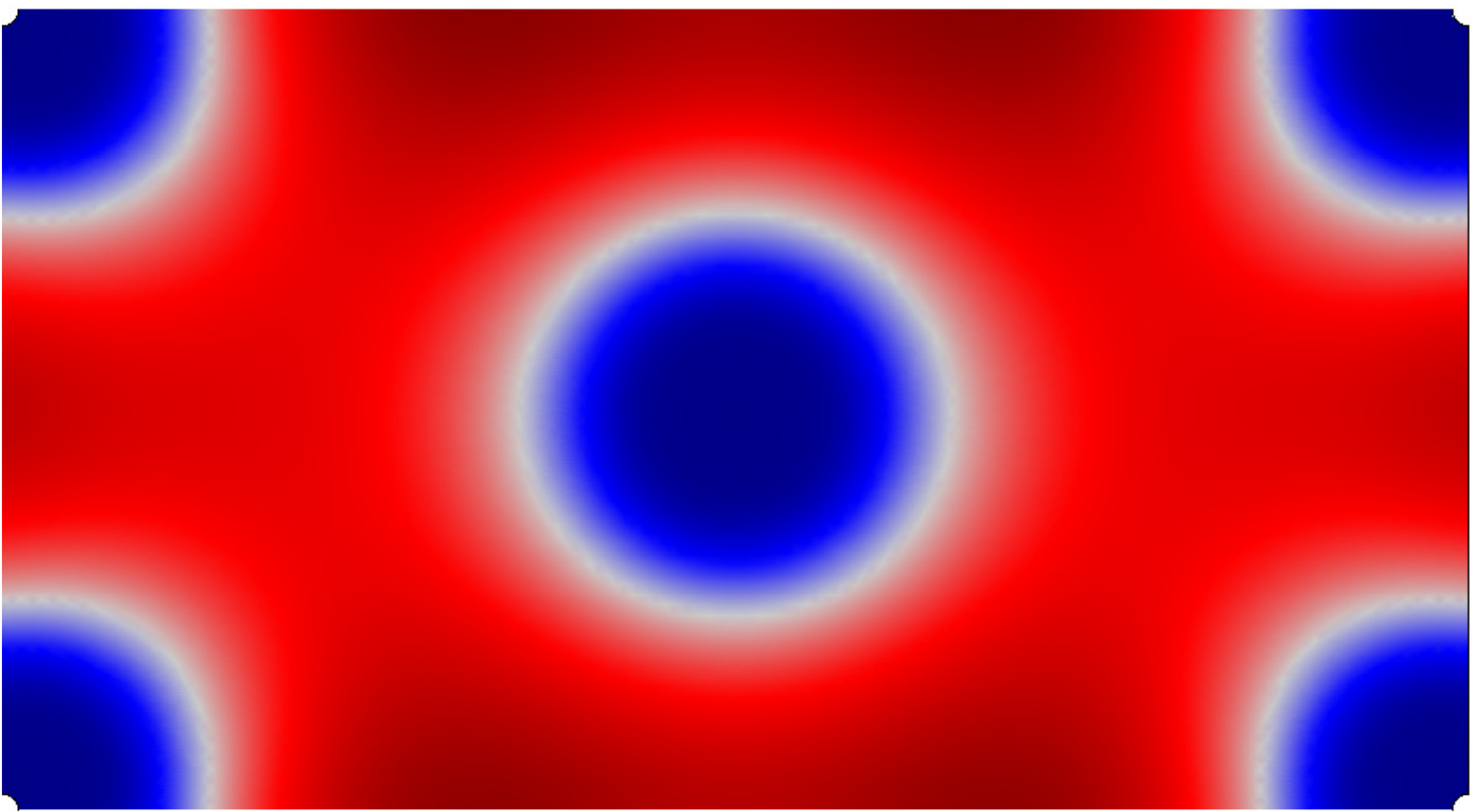}};
  \node (1) at (1*\pos, 0*\pos){\includegraphics[height=.25\textwidth, trim=0 0 00
      6, clip=true]{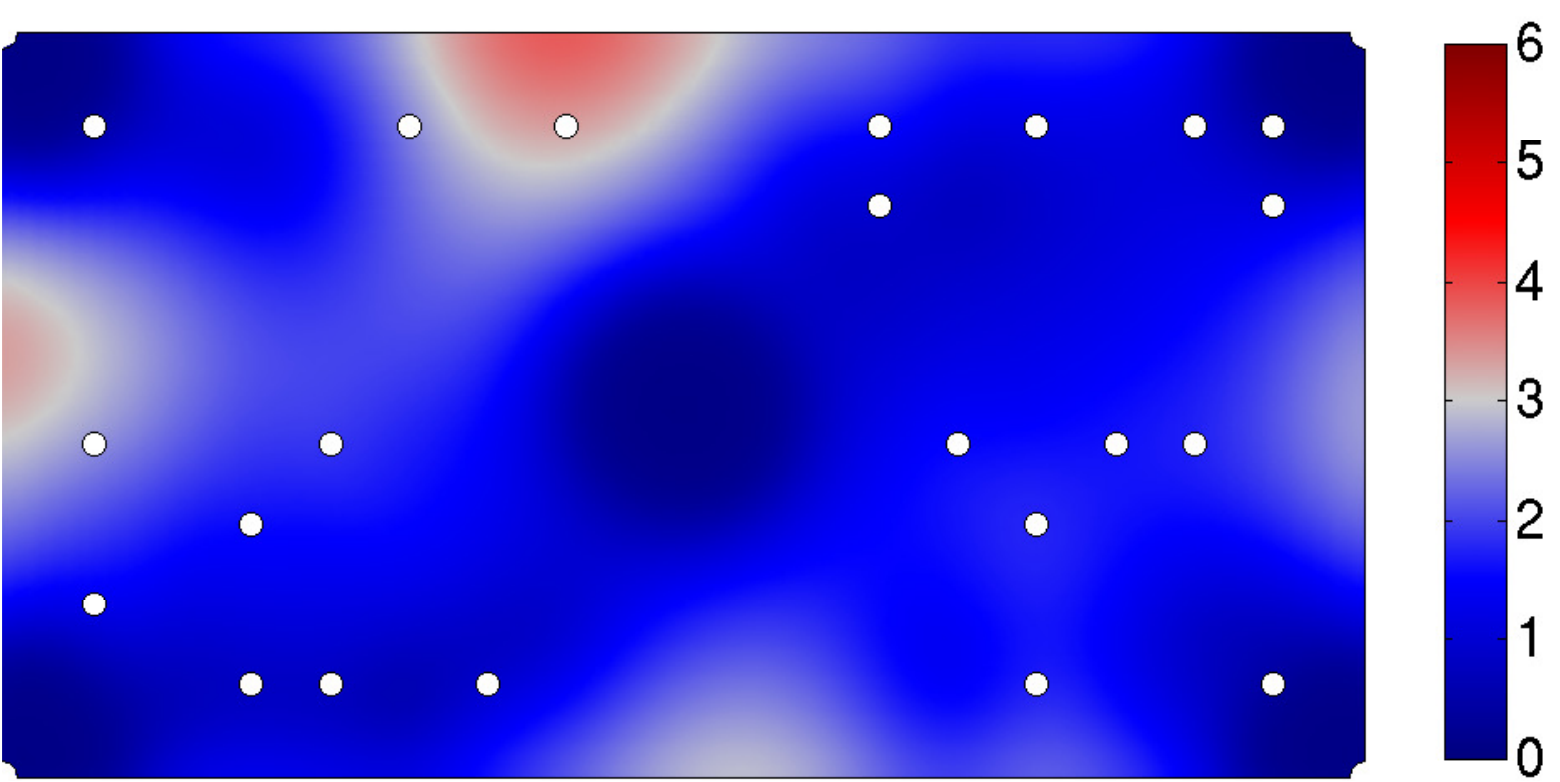}};
  \node at (0.05*\pos-0.48*\pos, 0.2*\pos) {\textcolor{white}{\sf a)}};
  \node at (1*\pos-0.49*\pos, 0.2*\pos) {\textcolor{white}{\sf b)}}; 
  \end{tikzpicture}
\caption{Shown are the prior standard deviation field (a) and the
  posterior standard deviation field based on solving the Bayesian
  inverse problem using the A-optimal sensor placement (b). The white
  dots in (b) indicate the A-optimal sensor locations.}
\label{fig:spe_aopt}
\end{figure}
Finally, to assess the effectiveness of the A-optimal sensor placement computed, we compare
the relative error of the MAP point as well as the average posterior variance, based on 
solving the Bayesian inverse problem using the optimal design versus that of solving the 
problem with randomly generated designs with the same number of sensors. 
Note that %
the A-optimal
sensor placement outperforms the random designs. 
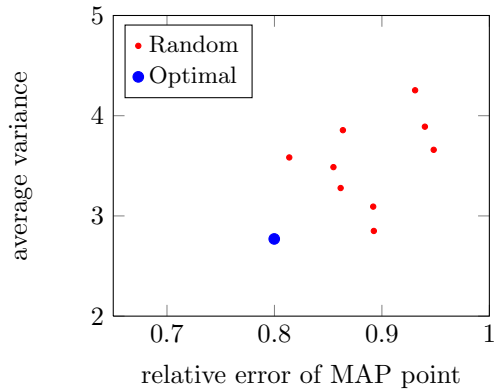
\begin{figure}[ht]\centering
\begin{tikzpicture}[]
\begin{axis}[width=5cm, height=4cm, scale only axis,
    xlabel = relative error of MAP point, ylabel=average variance, xmin=0.65, xmax=1,
    ymin=2, ymax=5, legend style={font=\small,nodes=right}, legend pos= north west]
\addplot [color=red, mark=*, only marks, mark size=1pt] table[x=err,y=avgvar]{./spe_cloud_randpts.txt};
\addlegendentry{Random}
\addplot [color=blue, mark=*, only marks, mark size=2pt] table[x=err,y=avgvar]{./spe_cloud_opt.txt};
\addlegendentry{Optimal}
\end{axis}
\end{tikzpicture}
\caption{Relative error of the MAP point versus
  average posterior variance for random designs (red dots) and for the
  optimal design (blue dot). The results corresponds to designs
  with $22$ sensors.}
\end{figure}

\section{Conclusions and remarks}

We have developed a scalable method for computing A-optimal
experimental designs for infinite-dimensional Bayesian nonlinear
inverse problems governed by PDEs. By scalable, we mean that the cost
(measured in forward-like PDE solves) of solving the OED problem is
independent of the parameter and sensor dimensions. The OED
formulation results in a bilevel optimization problem that features an
inverse problem as the inner optimization problem, and additional
forward-like PDEs representing the action of the inverse Hessian of
the inverse problem as constraints for the outer optimization problem.
We specialize this OED formulation to the problem of determining the
sensor placement that optimally infers the coefficient of an elliptic
PDE in the sense that the uncertainty in the recovered coefficient is
minimized over a set of prior model samples. For the resulting
PDE-constrained OED problem, we derive adjoint-based expressions for
the gradient, which enables use of efficient gradient-based
optimization algorithms. Computing the gradient of the OED objective
function requires differentiating expressions involving the Hessian,
which requires third derivatives of the parameter-to-observable
map. These are made tractable via a variational formulation of the OED
problem.
Numerical studies of the performance of our OED method for the
inference of the log permeability field in a porous media flow problem
indicate that the computational cost of computing an A-optimal
experimental design, measured in the number of forward-like PDE
solves, is insensitive to the dimension of the discretized parameter
field and to the sensor dimension.

A potential limitation of our method is defining the OED objective in
terms of a Gaussian approximation to the posterior distribution of the
parameter field.  However, as mentioned in the introduction, a
Gaussian provides a good approximation to the posterior in cases where
a linear approximation to the parameter-to-observable map over the set
of parameters with significant posterior probability is sufficiently
accurate.  Relaxing the Gaussian approximation of the posterior for
large-scale Bayesian inverse problems with expensive-to-evaluate
parameter-to-observable maps is extremely challenging.  The fact that
the Bayesian inverse problem is merely an inner problem for computing
OEDs compounds these challenges. 

A related consideration is the influence of the prior on the
OED obtained from our formulation. In cases where one has
limited prior information, samples from the prior may
have rather different features.
Since data computed from these vastly different prior samples are used as
``training data'' in our OED formulation, the resulting design might
be suboptimal for the ``truth'' parameter as we are searching for an
A-optimal design that accommodates a wide range of data.
In such cases, an effective  strategy could be an iterative process: namely, one
conducts initial field experiments and obtains a Bayesian
update, which better constrains the uncertain parameter field. This field is
then used as prior in the computation of an OED, whose target is to
collect additional experimental
data.

Another limitation of our approach is that our
sparsification strategy provides only indirect control on the number
of sensors in the optimal configuration. In practice, solving multiple
OED problems may be required to determine an appropriate penalty
parameter experimentally. This, however, is the price we pay to render
an otherwise combinatorial sensor placement problem computationally
tractable.

Computing optimal experimental designs still requires a large number
of forward (or adjoint or incremental) PDE solves. However, as
discussed in Section~\ref{sec:complexity}, a number of systems
characterized by the same Hessian operator must be solved at each OED
step, which suggests that using low rank Hessian approximations as
discussed
in~\cite{AlexanderianPetraStadlerEtAl14,Bui-ThanhGhattasMartinEtAl13,FlathWilcoxAkcelikEtAl11}
can mitigate this computational cost. Moreover, our OED method
contains important coarse-grained parallelism: the inverse problems
corresponding to each data sample can be solved independently.

In future work, we intend to study the sensitivity of the optimal
sensor placement to the number of data samples in the OED problem. The
data samples are generated by sampling the prior model; their number
is dictated by the need to solve an additional inverse problem for
each sample at each OED iteration. For this reason, the numerical
experiments in this paper have been limited to a small number of data
samples. However, we speculate that increasing the number of data
samples leads to diminishing returns, since the goal is not to fully
sample the prior, but to determine optimal sensor locations, and we
expect that they will be sensitive to only a limited number of
directions in the parameter space. Thus, an interesting extension of
this work is to determine how many data samples are needed. 

\appendix
\section{An infinite-dimensional trace estimator}\label{apdx:trace_estimator}
Let $\mu_\delta = \GM{0}{\C_\delta}$ and $\tilde{\mu}_\delta = \GM{0}{\A^{1/2} \C_\delta \A^{1/2}}$
with $\A$ and $\C_\delta$ as in the paragraph preceding~\eqref{equ:quadform}; the final 
equality in~\eqref{equ:quadform} follows by noting that
\[
    \int_\hilb \ip{z}{\A z} \, \mu_\delta(dz)
                           \!=\! \int_\hilb \|\A^{1/2}z\|^2\, \, \mu_\delta(dz)
                           \!=\! \int_\hilb \norm{y}^2 \, \tilde{\mu}_\delta(dy)
                           \!=\! \trace(\A^{1/2} \C_\delta\A^{1/2}) = \trace(\A\C_\delta).
\]
The following result justifies taking the limit as we let $\delta \to 0$.
\begin{proposition}
Let $\D$ be a bounded domain with Lipschitz boundary and consider the operator $\C_\delta = (-\delta \Delta + I)^{-2}$ defined on $L^2(\D)$, 
where $\delta$ is a positive real and $\Delta$ is the Laplacian operator on $\D$ with the natural boundary condition.
Suppose $\A$ is a positive self-adjoint trace-class operator on $L^2(\D)$. Then,
\begin{equation*}%
    \lim_{\delta \to 0} \trace(\A\C_\delta) = \trace(\A). 
\end{equation*}
\end{proposition}
\begin{proof}
Let us consider the difference, 
$\trace(\A) - \trace(\A\C_\delta) = \trace(\A(I - \C_\delta))$.
Denote by $\{e_i\}_{i = 1}^\infty$ the eigenvectors of $I - \C_\delta$ (independent of $\delta$) and by 
$\lambda_i^\delta$ the respective eigenvalues. By the definition of the operator $\C_\delta$ 
we have $\lambda_i^\delta = 1 - 1/(1 + \delta \nu_i)^2$
where $\nu_i$ are the (unbounded) eigenvalues of $-\Delta$. 
Using the fact that $0 \leq \nu_i \to \infty$, we know $0 \leq \lambda_i^\delta < 1$ for all $\delta > 0$. 
Next, we note
\begin{equation*}
   \trace(\A(I - \C_\delta)) = \sum_{i = 1}^\infty \ip{e_i}{\A(I - \C_\delta)e_i} 
   = \sum_{i = 1}^\infty \lambda_i^\delta \ip{e_i}{\A e_i} < \infty.
\end{equation*}
Let $\eps > 0$ be fixed but arbitrary and note that we can fix $N_0 \in \N$ such that, 
$
   \sum_{i = N_0 + 1}^\infty \lambda_i^\delta \ip{e_i}{\A e_i} \leq \sum_{i = N_0 + 1}^\infty \ip{e_i}{\A e_i} < \eps/2
$.
Also, we can choose $\delta > 0$ sufficiently small so that,
$
   \sum_{i = 1}^{N_0} \lambda_i^\delta \ip{e_i}{\A e_i} \leq \norm{A} \sum_{i = 1}^{N_0} \lambda_i^\delta < \eps/2, 
$
and hence the assertion of the proposition follows. 
\end{proof}

\section{Gradient derivation of OED objective function $\hat\obj$}\label{appdx:oed-gradient}
Here, we summarize the derivation of the gradient of the OED objective
function presented in~\eqref{equ:oed-grad}. 
To derive the expression for the gradient, we employ a
formal Lagrangian approach \cite{Troltzsch10}, which uses a Lagrangian
function composed of the objective function \eqref{equ:outeropt} with
the PDE constraints \eqref{equ:state}--\eqref{equ:incgrad} enforced
through Lagrange multiplier functions.  This Lagrangian function
$\LOED$ for the OED problem is given by:
\begin{alignat*}{1}
  \LOED&\left(\vec{w}, \{u_i\}, \{m_i\}, \{p_i\}, \{v_{ik}\}, \{q_{ik}\},\{ y_{ik}\}, 
              \{\ad{u}_i\}, \{\ad{m}_i\}, \{\ad{p}_i\}, \{\ad{v}_{ik}\}, \{\ad{q}_{ik}\}, \{\ad{y}_{ik}\}\right) \\
  =&\frac{1}{\Nd\Ntr}\sum_{i=1}^\Nd \sum_{k = 1}^\Ntr \ip{ z_k }{y_{ik}} \\
  &+ \sum_{i = 1}^\Nd \big[ \ip{\Exp{m_i} \grad u_i}{\grad \ad{u}_i} -  \ip{f}{\ad{u}_i} - \ip{h}{\ad{u}_i}_{\GN}\big]\\
  &+ \sum_{i = 1}^\Nd \big[\ip{\Exp{m_i} \grad p_i}{\grad \ad{p}_i} + \ip{\B^*\Wn(\B u_i - \obs_i)}{\ad{p}_i}\big]\\
  &+ \sum_{i = 1}^\Nd \big[\cip{m_i - \iparpr}{ \ad{m}_i} + \ip{\ad{m}_i \Exp{m_i}\grad u_i}{\grad p_i}\big]\\
  &+ \sum_{i = 1}^\Nd \sum_{k = 1}^\Ntr\big[\ip{ \Exp{m_i}\grad v_{ik}}{\grad \ad{v}_{ik}} + \ip{y_{ik}\Exp{m_i}\grad{u_i}}{\grad \ad{v}_{ik}}\big]\\
  &+ \sum_{i = 1}^\Nd \sum_{k = 1}^\Ntr\big[\ip{ \Exp{m_i}\grad q_{ik}}{\grad \ad q_{ik}} + \ip{y_{ik} \Exp{m_i}\grad p_i}{\grad \ad{q}_{ik}} + \ip{\B^*\Wn\B v_{ik}}{\ad{q}_{ik}}\big]\\
  &+ \sum_{i = 1}^\Nd \sum_{k = 1}^\Ntr\big[\ip{\ad{y}_{ik} \Exp{m_i}\grad v_{ik}}{\grad p_i} + \cip{\ad{y}_{ik}}{y_{ik}}
  + \ip{\ad{y}_{ik}\Exp{m_i}\grad u_i}{\grad q_{ik}}\\
  &\qquad\qquad+ \ip{\ad y_{ik} y_{ik} \Exp{m_i} \grad u_i}{\grad p_i} - \ip{z_k}{\ad{y}_{ik}}\big].
\end{alignat*}
The variables $(u_i, m_i, p_i) \in \Vg\times \CM\times \V$, for $i \in
\{1, \ldots,\Nd\}$, and $(v_{ik}, q_{ik}, y_{ik})\in \V\times \V \times
\CM$, with $(i, k) \in \{1,
\ldots, \Nd\} \times \{1, \ldots, \Ntr\}$ are the \emph{OED state
  variables}.
The \emph{OED adjoint variables}
$\ad{u}_i, \ad{m}_i, \ad{p}_i, \ad{v}_{ik}, \ad{q}_{ik}$, and
$\ad{y}_{ik}$ belong to the
test function spaces corresponding to their state counterparts.

The gradient for~\eqref{equ:outeropt} is given by the derivative of
$\LOED$ with respect to the weight vector $\vec w$, provided that
variations of $\LOED$ with respect to the OED state and adjoint variables vanish. 
The weight vector
enters the Lagrangian through the weight matrix 
$\Wn = \sum_{j = 1}^\Ns w_j \mat{E}_j$, where  $\mat{E}_j =
\sigma_j^{-2} \vec{e}_j \vec{e}_j^T$. (Here $\vec{e}_j$ denotes the $j$th 
standard basis vector in $\R^\Ns$.) 
Using this notation, it is straightforward to compute derivatives of
the Lagrangian function with
respect to $w_j$, the $j$th component of the weight vector $\vec w$:
\[
\LOED_{w_j} = \sum_{i = 1}^\Nd \ip{\B^* \mat{E}_j (\B u_i - \obs_i)}{ \ad{p}_i} + 
\sum_{i=1}^\Nd \sum_{k = 1}^\Ntr \ip{\B^* \mat{E}_j \B v_{ik}}{\ad{q}_{ik}}, \quad \mbox{ for } j = 1, \ldots, \Ns.
\]
Recalling the definition of $\mat{E}_j$ and using a vector form for
the gradient, we obtain
\begin{equation}\label{equ:gradw}
\hat\obj' = 
    \sum_{i = 1}^\Nd \ncov^{-1}(\B u_i - \obs_i) \odot \B\ad{p}_i + 
    \sum_{i = 1}^\Nd \sum_{k = 1}^\Ntr \ncov^{-1} \B v_{ik} \odot \B \ad{q}_{ik},
\end{equation}
provided
appropriate state and adjoint equations are satisfied. These
equations are computed next.

Requiring that variations of $\LOED$ with respect to the 
OED adjoint variables vanish, we recover the OED state
  equations \eqref{equ:state}--\eqref{equ:incgrad}.
The variables $\ad{p}_{ik}$ and $\ad{q}_{ik}$ are defined through
adjoint equations, obtained by requiring that variations of
$\LOED$ with respect to the OED state variables vanish. That is, for
each $i \in \{1, \ldots, \Nd\}$ and $k \in \{1, \ldots, \Ntr\}$,
   \begin{align}
   \LOED_{v_{ik}}[\ut{v}] &= \ip{\B^*\Wn\B\ut{v}}{\ad{q}_{ik}} + \ip{\ad{y}_{ik} \Exp{m_i}\grad \ut{v}}{\grad p_i} + \ip{\Exp{m_i}\grad\ut{v}}{\grad \ad{v}_{ik}} = 0,
   \label{equ:outer-adj1}
   \\
   \LOED_{q_{ik}}[\ut{q}] &= \ip{\Exp{m_i} \grad \ut q}{\grad \ad q_{ik}} \!+\! \ip{\ad{y}_{ik} \Exp{m_i}\grad u_i}{\grad \ut{q}} \!=\! 0,
   \label{equ:outer-adj2}
   \\
   \LOED_{y_{ik}}[\ut{y}] &= \ip{\ut{y} \Exp{m_i}\grad p_i}{\grad \ad{q}_{ik}}
   \!+\! \cip{\ad{y}_{ik}}{\ut{y}} \!+\! \ip{\ut y \ad y_{ik} \Exp{m_i} \grad u_i}{\grad p_i}
   \!+\! \ip{\ut{y} \Exp{m_i}\grad u_{i}}{\grad \ad{v}_{ik}}\nonumber 
   \\&\tab\tab\!+\!\frac1{\Nd\Ntr}\ip{z_k}{\ut{y}} \!=\! 0,
   \label{equ:outer-adj3}
   \\
   \LOED_{u_i}[\ut u] &= 
               \ip{\B^*\Wn\B\ut u}{\ad p_i}
               \!+\! \ip{\ad{m}_i \Exp{m_i}\grad \ut{u}}{\grad p_i}
               \!+\! \ip{\Exp{m_i} \grad \ut{u}}{\grad \ad{u}_i}
               \!-\! \ip{b^{(1)}_i}{\ut u} \!=\! 0, 
   \label{equ:outer-adj4}
   \\ 
   \LOED_{m_i}[\ut m] &= 
               \ip{\ut{m}\Exp{m_i}\grad{p}_i}{\grad\ad{p}_i}
               \!+\! \cip{\ad{m}_i}{\ut{m}} \!+\! \ip{\ut m \ad{m}_i \Exp{m_i}\grad u_i}{\grad p_i}
               \nonumber\\
               &\tab\tab\!+\! \ip{\ut{m}\Exp{m_i}\grad u_i}{\grad\ad{u}_i}
               \!-\! \ip{b_i^{(2)}}{\ut m} \!=\! 0, 
   \label{equ:outer-adj5}
   \\ 
   \LOED_{p_i}[\ut p] &= 
               \ip{\Exp{m_i} \grad \ut{p}}{\grad \ad{p}_i}
               \!+\! \ip{\ad{m}_i \Exp{m_i}\grad u_i}{\grad \ut{p}}
               \!-\! \ip{b_i^{(3)}}{\ut p} \!=\! 0,  
   \label{equ:outer-adj6}
   \end{align}
for all $(\ut{v}, \ut{q}, \ut{y}, \ut{u}, \ut{m}, \ut{p}) \in \V \times \V \times \CM \times \V \times \CM \times \V$. 
Here, $b_i^{(1)}$, $b_i^{(2)}$, and $b_i^{(3)}$ are
\begin{equation}\label{equ:rhs-ugly}
\begin{aligned}
   \ip{b^{(1)}_i}{\ut u} = &-\sum_{k = 1}^\Ntr \big[\ip{y_{ik} \Exp{m_i} \grad\ut{u}}{\grad \ad v_{ik}}
  +\ip{\ad{y}_{ik} \Exp{m_i} \grad \ut{u}}{\grad q_{ik}}
  +\ip{\ad y_{ik} y_{ik} \Exp{m_i} \grad \ut u}{\grad p_i}\big],
\\
   \ip{b^{(2)}_i}{\ut m} = &-\sum_{k = 1}^\Ntr \big[\ip{ \ut{m} \Exp{m_i} \grad v_{ik}}{\grad \ad v_{ik}}
  +\ip{\ut{m} \Exp{m_i} \grad q_{ik}}{\grad \ad{q}_{ik}}
  +\ip{\ut m y_{ik} \Exp{m_i} \grad u_i}{\grad \ad v_{ik}}
\\
 &\tab\tab\tab+\ip{\ut m y_{ik} \Exp{m_i} \grad  p_i}{\grad \ad q_{ik}}
   +\ip{\ut m \ad y_{ik} \Exp{m_i} \grad v_{ik}}{\grad p_i}
   +\ip{\ut m \ad y_{ik} \Exp{m_i} \grad u_i}{\grad q_{ik}}
\\
  &\tab\tab\tab+\ip{\ut m \ad y_{ik} y_{ik} \Exp{m_i} \grad u_i}{\grad p_i}\big],
\\
   \ip{b^{(3)}_i}{\ut p} = &-\sum_{k = 1}^\Ntr \big[\ip{y_{ik}\Exp{m_i}\grad\ut{p}}{\grad\ad{q}_{ik}}
  +\ip{\ad{y}_{ik} \Exp{m_i}\grad v_{ik}}{\grad \ut{p}}
  +\ip{\ad y_{ik} y_{ik} \Exp{m_i} \grad u_i}{\grad \ut p}\big].
\end{aligned}
\end{equation}
Upon inspecting the OED adjoint
equations~\eqref{equ:outer-adj1}--\eqref{equ:outer-adj6} and comparing
them to the system of
equations~\eqref{equ:state}--\eqref{equ:incgrad}, we notice that the
OED adjoint equations inherit structure from the OED state
equations. Specifically, notice that after rearranging and identifying
terms, the system~\eqref{equ:outer-adj1}--\eqref{equ:outer-adj3} for
$(\ad{q}_{ik}, \ad{v}_{ik}, \ad{y}_{ik})$ is the same as the system
\eqref{equ:incstate}--\eqref{equ:incgrad}, except for the
right hand sides, which coincide up to a constant. This reveals the
following relations:
\begin{equation*}%
    \ad{q}_{ik} = -\frac{1}{\Nd\Ntr}v_{ik}, \quad \ad{y}_{ik} =
    -\frac{1}{\Nd\Ntr}y_{ik}, \quad \ad{v}_{ik} =
    -\frac{1}{\Nd\Ntr}q_{ik},
\end{equation*}
for $i \in  \{1, \ldots, \Nd\}$ and $k \in \{1, \ldots, \Ntr\}$.
Thus, the OED adjoint variables $\ad{q}_{ik}$, $\ad{y}_{ik}$, and $\ad{v}_{ik}$ 
can be eliminated from the system
and the right hand
sides $b^{(1)}_{ik}, b^{(2)}_{ik}, b^{(3)}_{ik}$ 
defined in~\eqref{equ:rhs-ugly} simplify, and result in \eqref{equ:rhs-nice}.

\section{Discretization and computational details}\label{appdx:discretization}
We use a finite-element discretization of the parameter field and the
state and adjoint variables, and we denote by 
boldfaced letters the discretized versions of the variables and
operators appearing in the expressions. 
Next, we describe the numerical computation of the OED objective
function in~\eqref{equ:outeropt}
and of its gradient, where we again consider that $\upgamma = 0$. The
discrete OED function is
\begin{equation}\label{equ:oed-obj}
\hat\obj_h(\vec{w}) = \frac{1}{\Nd\Ntr} \sum_{i = 1}^\Nd \sum_{k = 1}^\Ntr \mip{\vec{z}_k}{\vec{y}_{ik}}.
\end{equation}
Note that, to discretize the infinite-dimensional Hilbert space, we use a
mass-weighted inner product in \eqref{equ:oed-obj}. This is necessary
since the finite-dimensional inference parameters are the coefficients
of the finite element approximation, and helps to ensure that the discrete
problems are appropriate discretizations of 
the infinite-dimensional problem.
We rely on a Gaussian trace estimator and let  
$\vec{z}_k = \mat{M}^{-1/2} \vec{\nu}_k$, $k = 1, \ldots, \Ntr$,
where $\vec{\nu}_k$ are draws from $\GM{\vec{0}}{\vec{I}}$.
See~\cite{AlexanderianPetraStadlerEtAl14} for a justification of the
form of the mass-weighted trace estimator and also
an efficient procedure for computing the application of
$\mat{M}^{-1/2}$ to a vector.

For a given design $\vec{w}$ and data samples $\obs_i$, $i \in \{1, \ldots, \Nd\}$, we solve the inner optimization 
problem~\eqref{equ:state}--\eqref{equ:grad} for the MAP point
$\dpar_i = \dparmap(\vec{w}; \obs_i)$; we also evaluate the state $\vec{u}_i$ and adjoint $\vec{p}_i$ 
variables (for the inner optimization) at the MAP point. Next, we need to solve for $\vec{y}_{ik}$ and the variables $\vec{v}_{ik}$ and $\vec{q}_{ik}$ in~\eqref{equ:incstate}--\eqref{equ:incgrad}.
This is accomplished by solving a linear system of the following block form 
\newcommand{\Luu}{\mat{D}}
\newcommand{\Lum}{\mat{S}^T}
\newcommand{\Lmu}{\mat{S}}
\newcommand{\Lmm}{\mat{Q}}
\newcommand{\Lmmhat}{\widehat{\Lmm}}

\begin{equation} \label{equ:KKTy-fd}
\begin{bmatrix}
\Luu & \Lum & \mat{A}^T\\
\Lmu & \Lmm   & \mat{C}^T\\  
\mat{A}      & \mat{C}      & \mat{0}  
\end{bmatrix}
\begin{bmatrix}
\vec{v}_{ik}\\
\vec{y}_{ik}\\
\vec{q}_{ik}
\end{bmatrix}
=
\begin{bmatrix}
\vec{0} \\
\vec{z}_k \\
\vec{0}
\end{bmatrix}. 
\end{equation}
In the above system, $\Luu = \mat{B}^T \Wn \mat{B}$, where $\mat{B}$ is the
discretization of the observation operator $\B$. The remaining blocks in the system are 
discretizations of the differential operators appearing
in~\eqref{equ:incadjoint}--\eqref{equ:incstate}, evaluated at
$(\vec{u}_i, \dpar_i, \vec{p}_i)$; we refer to
\cite{PetraStadler11} for more details on the discretization of the
Hessian system for an inverse coefficient problem
with an elliptic PDE.
To solve the system~\eqref{equ:KKTy-fd}, we first block eliminate $\vec{v}_{ik}$ and $\vec{q}_{ik}$, namely
\begin{align*}%
    \vec{v}_{ik} = -\mat{A}^{-1}\mat{C} \vec{y}_{ik}, \quad
    \vec{q}_{ik} = -\mat{A}^{-T} (\Luu \vec{v}_{ik} + \Lum \vec{y}_{ik}),
\end{align*}
for $i \in \{1, \ldots, \Nd\}$ and $k \in \{1, \ldots, \Ntr\}$, and solve $\mat{H} \vec{y}_{ik} = \vec{z}_{k}$
with 
\begin{equation}\label{equ:reduced-hess}
\mat{H} = \mat{C}^T \mat{A}^{-T} (\Luu \mat{A}^{-1} \mat{C} - \Lum) 
             - \Lmu \mat{A}^{-1} \mat{C} 
             + \Lmm.
\end{equation}
Once $\vec{y}_{ik}$ is available for $i \in \{1, \ldots, \Nd\}$ and $k \in \{1, \ldots, \Ntr\}$,
we can compute the OED objective function~\eqref{equ:oed-obj}.

To compute the gradient we also need the OED adjoint 
variables $\adfd{p}_i$, $i = 1, \ldots, \Nd$, which are computed
by solving a linear system similar to~\eqref{equ:KKTy-fd}, for $(\adfd{p}_i, \adfd{m}_i, \adfd{u}_i)$, where the
blocks in the system right hand side are replaced by  
$\vec{b}^{(1)}_i, \vec{b}^{(2)}_i$, and $\vec{b}^{(3)}_i$ which are discretizations of the expressions in~\eqref{equ:rhs-nice}.
\commentout{
The operators appearing in the right-hand-side are defined as follows:
The matrix $\mat{A}_{\alpha}$ is the discretiztion of the operator,
\begin{equation*}%
    \A_{\alpha}[u] = -\grad \cdot \big(\alpha \Exp{m} \grad u\big),
\end{equation*}
and $\mat E$ and $\mat F_i, i = \{1,2\}$ are the discretization of the operators
\begin{equation*}%
\begin{aligned}
        \E[v] &= \Exp{m}\grad v \cdot \grad q \\
        \F_1[y] &= y \Exp{m} \grad u \cdot \grad q \\
        \F_2[y] &= y \Exp{m} \grad v \cdot \grad p.
\end{aligned}
\end{equation*}
}
Thus, we solve $\mat{H} \adfd{m}_i = \bar{\vec{b}}_i$,  
where $\mat{H}$ is as in~\eqref{equ:reduced-hess}, and $\bar{\vec{b}}_i$ is given by
\begin{equation*}%
\bar{\vec{b}}_i = \vec{b}^{(2)}_i 
                - \mat{C}^T \mat{A}^{-T}\vec{b}^{(1)}_i 
                - \Lmu\mat{A}^{-1} \vec{b}^{(3)}_i
                + \mat{C}^T \mat{A}^{-T}\Luu \mat{A}^{-1} \vec{b}^{(3)}_i,  
\end{equation*}
for $i \in \{1, \ldots, \Nd\}$. 
Next, we solve for $\adfd{p}_i$, 
\begin{equation*}%
\adfd{p}_i = \mat{A}^{-1} (\vec{b}^{(3)}_i - \mat{C} \adfd{m}_i), \quad i \in \{1, \ldots, \Nd\}.
\end{equation*}
Subsequently, we have all the quantities required in the expression
for the (discretized) gradient:
\begin{equation*}%
    \nabla \hat \obj_h(\vec{w}) = \sum_{i = 1}^\Nd \ncov^{-1}(\mat{B} \vec{u}_i - \obs_i) \odot \mat{B}\adfd{p}_i - 
    \frac{1}{\Nd\Ntr} \sum_{i = 1}^\Nd \sum_{k = 1}^\Ntr \ncov^{-1}\mat{B} \vec{v}_{ik} \odot \mat{B} \vec{v}_{ik}. 
\end{equation*}

\bibliographystyle{siam}

\end{document}